\documentclass{article}
\usepackage{graphicx} 

\usepackage[margin=1in]{geometry}
\usepackage{subfig}
\usepackage{float}
\usepackage{enumitem}

\usepackage{color,soul}
\usepackage{ulem}
\usepackage{comment}

\usepackage{todonotes}
\usepackage{amsfonts}
\usepackage{amsmath}
\usepackage{amssymb}
\usepackage{amsthm}

\usepackage{algorithm}
\usepackage{algpseudocode}

\usepackage{hyperref}

\newcommand{\R}{\mathbb{R}}
\newcommand{\N}{\mathbb{N}}
\newcommand{\calR}{\mathcal{R}}
\newcommand{\calX}{\mathcal{X}}
\newcommand{\calP}{\mathcal{P}}
\renewcommand{\L}{\mathcal{L}}
\newcommand{\E}{\mathbb{E}}

\newcommand{\Tr}{\operatorname{Tr}}
\newcommand{\diag}{\operatorname{diag}}
\newcommand{\x}{\boldsymbol{x}}

\newtheorem{theorem}{Theorem}
\newtheorem{assumption}{Assumption}
\newtheorem{remark}{Remark}
\newtheorem{lemma}{Lemma}
\newtheorem{corollary}{Corollary}

\usepackage{orcidlink}

\title{Convergence and Sketching-Based Efficient Computation of Neural Tangent Kernel Weights in Physics-Based Loss}
\author{Max Hirsch\,\orcidlink{0009-0003-6496-6664}\\Department of Mathematics\\University of California, Berkeley\\Berkeley, CA 94720, USA\\\texttt{mhirsch@berkeley.edu} \and Federico Pichi\,\orcidlink{0000-0002-1163-3386}\\mathLab, Mathematics Area\\SISSA\\I-34136 Trieste, Italy\\\texttt{fpichi@sissa.it}}
\date{\today}

\begin{document}

\maketitle

\begin{abstract}
    In multi-objective optimization, multiple loss terms are weighted and added together to form a single objective. These weights are chosen to properly balance the competing losses according to some meta-goal. For example, in physics-informed neural networks (PINNs), these weights are often adaptively chosen to improve the network's generalization error. A popular choice of adaptive weights is based on the neural tangent kernel (NTK) of the PINN, which describes the evolution of the network in predictor space during training. The convergence of such an adaptive weighting algorithm is not clear \textit{a priori}. Moreover, these NTK-based weights would be updated frequently during training, further increasing the computational burden of the learning process. In this paper, we prove that under appropriate conditions, gradient descent enhanced with adaptive NTK-based weights is convergent in a suitable sense. We then address the problem of computational efficiency by developing a randomized algorithm inspired by a predictor-corrector approach and matrix sketching, which produces unbiased estimates of the NTK up to an arbitrarily small discretization error. Finally, we provide numerical experiments to support our theoretical findings and to show the efficacy of our randomized algorithm.
\end{abstract}

\begin{center}
    \textbf{Code Availability: \hyperlink{https://github.com/maxhirsch/Efficient-NTK}{https://github.com/maxhirsch/Efficient-NTK}} 
\end{center}

\section{Introduction}

In multi-objective optimization, the goal is to simultaneously minimize multiple competing loss objectives $\{\L_i(\theta)\}_{i=1}^n$  w.r.t.\ some parameter $\theta$ \cite{collette2004multiobjective,YangMultiobjectiveChapter}. The objectives $\L_i(\theta)$ may not be minimized at the same $\bar\theta$, so it is often impossible to minimize each objective individually at the same time. In general, one can only find a solution in a set of Pareto solutions for which none of the objectives can be decreased without another objective increasing. One approach to finding such a solution to this problem is the so-called \textit{weighted sum method} \cite{Marler2010,Kim2005AdaptiveWeightedSum}. 

In the weighted sum method, one attempts to reach the goal by minimizing the weighted objective
\[
    \L(\theta) = \sum_{i=1}^n w_i\L_i(\theta)
\]
for some weights $w_i \in [0,\infty)$. Choosing a weight $w_i$ that is larger than some $w_j$ will encourage the objective $\L_i(\theta)$ to be minimized more than $\L_j(\theta)$, and modifying the weights results in new minimizers being found. The weighted sum method is known to not be able to find all Pareto solutions in non-convex regions of the Pareto solution set, but the method is still widely used due to its simplicity \cite{Kim2005AdaptiveWeightedSum}. In particular, in the scientific machine learning (SciML) context \cite{SciMLBook,prince2023understanding,QuarteroniCombiningPhysicsbasedDatadriven2025a}, it is a common task when training physics-informed neural networks (PINNs).

A PINN is a method for solving partial differential equations (PDEs) using a neural network \cite{KarniadakisPINN,Karniadakis2021,Meng2025,hao2023physicsinformedmachinelearningsurvey}. The neural network takes as input the spatial (and possibly temporal) coordinates at which the PDE solution must be evaluated, and the output of the network is the PDE solution evaluated at that point. To train a PINN, there are typically a few separate objectives or loss terms depending on whether the PDE is time dependent. For steady problems, the loss objectives include the residual of the differential operator for the PDE (usually the strong form of the problem), and the residual of the boundary condition (when imposed via weak constraints). If the problem is time dependent, there is an additional term corresponding to the (weak) imposition of the initial condition. These losses are combined via the weighted sum method, giving a single loss function to train the PINN. However, choosing the weights for the loss function is known to be a difficult problem, so several methods have been developed to learn these weights adaptively during the training procedure \cite{WangNTK2022,GAO2025,XIANG2022,Wang2024,LogWeights}.

One such method of choosing adaptive weights is based on \textit{neural tangent kernel} (NTK) theory \cite{WangNTK2022}. In that paper, it was shown, and observed for the Poisson equation, that as the width of a physics-informed neural network increases to infinity, the dynamics of the residuals used to define the PINN loss becomes linear. The matrix in the resulting linear ordinary differential equation for the residuals is called the neural tangent kernel. The singular values of the NTK thus determine the rates at which the different residuals, corresponding to the objective terms, decrease to zero. This inspired the authors of \cite{WangNTK2022} to devise a heuristic algorithm for choosing the weights in PINN loss based on the NTK of the PINN. The idea of the algorithm is to weight the losses roughly according to the inverses of the singular values of the NTK to make the rates at which the residuals decrease nearly the same. Furthermore, since in a practical setting the neural network is not infinitely wide, the computed NTK is not constant throughout epochs, requiring the NTK-based loss weights to be updated periodically during training.

This adaptive choice of weights in the loss function raises two main questions. First, does the optimization algorithm used to minimize the aggregate loss function converge when adaptive weights are used? Indeed, multi-objective problems optimized using the weighted sum method result in different minimizers of the objective when different weights are used. If these weights change frequently, it is not clear that an optimization algorithm with adaptive loss weights will ever converge to a single minimizer. Second, can the NTK weight updates be made more efficiently? Computing the NTK is a very costly operation since the matrix multiplication needed to compute a single entry of the NTK matrix of a neural network with $O$ outputs and $P$ parameters has computational complexity $\mathcal{O}(O^2P)$, and $P$ is typically large for modern neural network models, with $P\sim 10^7$ \cite{FastFiniteWidthNTK}. Thus, recomputing the exact weights can hinder the whole training phase by making it unbearable. However, based on the heuristics of the NTK-based weighting algorithm, it seems to be desirable to keep the weights updated. To do so, it is important to find faster ways of approximating the NTK-based weights.

We address each of these two questions. For the first, we apply a well-known ``descent lemma" for non-convex optimization \cite{WardDescentLemma, nesterov2004introductory,StochasticProgramming} to a setting in which a squared error loss is computed with respect to an adaptive inner product. This allows us to show that the average squared norm of the loss gradients using the adaptive NTK weights converges to zero. Then, to address the second question, we develop a simple approach inspired from predictor-corrector methods to obtain fast matrix-vector products with the NTK.

In particular, in a ``predictor" step, we use an initial gradient descent step from $\theta_t$ to $\widehat{\theta}_{t+1}$ with a special choice of randomized weights, allowing us to approximate the product of the NTK with a Gaussian random vector. Combining this with ideas from matrix sketching, commonly used to approximate the range of a matrix with repeated randomized matrix-vector multiplications \cite{RandomizedLinearAlgebraSurvey}, we are able to obtain a fast randomized approximation of the NTK and its trace. We then use this approximation in a ``corrector" step, which consists of returning to the original gradient descent iterate $\theta_t$ and computing the NTK-based weights from our NTK approximation to finally compute $\theta_{t+1}$ via gradient descent.

\textit{Outline.}
The remainder of the manuscript is organized as follows: in Section \ref{sec:literature-review}, we present a literature review of related works. Next in Section \ref{sec:NTK}, we present the use of NTK-based weights in the loss function for training PINNs. In Section \ref{sec:convergence}, we prove convergence results for optimization of the PINN loss function involving adaptive weights. Then in Section \ref{sec:fast-NTK}, we introduce a method for obtaining a fast approximation of the NTK and to allow frequent updates of the adaptive PINN loss weights. Finally, we present some numerical applications in Section \ref{sec:numerics} that verify our theoretical results and demonstrate the efficacy of our fast NTK approximation to compute the adaptive loss weights.

\subsection{Literature Review}\label{sec:literature-review}

In the physics-informed machine learning community, there has been extensive work on the use of adaptive weights for PINN losses. In \cite{WangNTK2022}, the authors use weights for the loss based on the neural tangent kernel.
The paper \cite{GAO2025} develops a weighting algorithm that assigns larger loss weights to the largest objectives. 
The authors of \cite{XIANG2022} and \cite{Wang2024} use maximum likelihood-based approaches to compute the loss weights.
In \cite{WangMitigating}, the authors connect the difficulty of training PINNs to an imbalance in the magnitudes of the gradients of the different constituent loss functions, and develop adaptive loss weights based on the relative sizes of the various loss gradients.
The paper \cite{LogWeights} also presents a variant of the gradient-based weights combined with the logarithmic mean. 
Despite the great body of literature in this important topic, up to our knowledge, there has not been any work showing convergence results for these types of loss weighting algorithms.

Recently, following the work \cite{WangNTK2022} in which the authors show convergence of the PINN NTK in the infinite-width limit, the neural tangent kernel has increasingly been used in the context of physics-informed neural networks, mainly for analyzing new algorithms and architectures. The paper \cite{PINNNTKGeneral} extends the analysis of the PINN NTK to more general PDEs, while \cite{SA-PINNs} uses the NTK to analyze a new scheme for adaptive loss weights. In \cite{GPINN-NTK}, the authors use the NTK to analyze a PINN trained with an additional loss for the gradient of the PDE residual. Finally, the authors of \cite{PIKANs} use the NTK to analyze physics-informed Kolmogorov-Arnold Networks.

Concerning the convergence of optimization algorithms for PINNs, so far there have been several different results. In \cite{PINNMinimizers}, the authors show for linear second-order elliptic and parabolic PDEs that, given a sequence of training sets for PINNs, the sequence of minimizers of the PINN loss converges to the PDE solution. In the setting of linear second-order PDEs, the paper \cite{OverparameterizedGDPINN} shows convergence of gradient descent for overparameterized two-layer PINNs to global minimizers of the loss, while \cite{OverparameterizedSGDPINN} considers a similar problem focused on stochastic gradient descent and Poisson's equation. In \cite{AppliedOptimizationPINN}, the authors present a more applied study of optimization for PINNs. Finally, \cite{PDEOrderPINNOptimization} considers the effect of PDE order on the PINN optimization.

There are several works on the fast estimation of the neural tangent kernel in different contexts. 
In \cite{FastFiniteWidthNTK}, the authors analyze the cost of exact computation of the NTK for finite-width neural networks.
The authors in \cite{FastApproximationENTK} approximate the full neural tangent kernel of a network using a smaller matrix independent of the dimension of the network output. 
Moreover, the authors of \cite{BayesianNetworksMonteCarlo} propose a Monte Carlo estimator of the NTK, while in \cite{FastNTKSketchRandomFeatures} the NTK is computed by combining sketching with random features.
Finally, there are also works on faster computation of graph neural tangent kernels for graph neural networks \cite{GNNReview,GNN-Intro} using sketching techniques \cite{FastGraphNTK}. 

\section{Neural Tangent Kernel}\label{sec:NTK}

We now describe the setting of physics-informed neural networks and their corresponding neural tangent kernels. In fact, the NTK of a PINN differs from that of a typical neural network because of the presence of a differential operator in the PINN loss function. Thus, we start by presenting the loss function of a general PINN before defining its NTK, along with the adaptive loss weighting algorithm from \cite{WangNTK2022}.

\subsection{Physics-Informed Neural Network Loss}

We consider the problem of finding $u_\theta : \Omega\subseteq\R^d\to\R^O$ parametrized by the network parameters $\theta\in\calP\subseteq\R^p$ which approximately solves the following general partial differential equation defined in the domain $\Omega$ w.r.t.\ the unknown $u$:
\[
    \begin{cases}
        \mathcal{D} u = f &\text{in }\Omega,\\
        \mathcal{B} u = g &\text{in }\partial\Omega,
    \end{cases}
\]
where $\mathcal{D}$ is a differential operator, $\mathcal{B}$ defines boundary (and possibly initial) conditions, $f:\Omega\to\R^O$ is a forcing term, and $g:\partial\Omega\to\R^O$ are boundary data. For simplicity of presentation, we assume $O=1$, but everything extends in a straightforward manner to multi-dimensional output, i.e.\ $O>1$ (e.g.\ the numerical example in Section \ref{ne:q-tensor} has $O=2$).

To be more precise in the formulation of the problem, we consider the collocation points $\calX_D = \{x_{D}^{i}\}_{i=1}^{n_D}$ and $\calX_B = \{x_{B}^{i}\}_{i=1}^{n_B}$ in $\Omega$, respectively for the interior and the boundary regions, with $n_D + n_B = n$ and $\calX = \calX_D\cup\calX_B$, and define the \textit{residual} as $\calR:\calP \to \R^n$ by
\[
    \calR(\theta)_i = \begin{cases}
        \mathcal{D}u_\theta(x_{D}^{i}) - f(x_{D}^{i}) &\text{if } 1\le i \le n_D,\\
        \mathcal{B}u_\theta(x_{B}^{i-n_D}) - g(x_{B}^{i-n_D}) &\text{if } n_D+1 \le i \le n.
    \end{cases}
\]
Then, we want to minimize the physics-informed loss function defined as the weighted sum
\begin{align*}
    \mathcal{L}(\theta) 
    &= \lambda_D\sum_{i=1}^{n_D} |\mathcal{D}u_\theta(x_{D}^{i}) - f(x_{D}^{i})|^2 + \lambda_B\sum_{i=1}^{n_B} |\mathcal{B}u_\theta(x_{B}^{i}) - g(x_{B}^{i})|^2\\ 
    &= \lambda_D\sum_{i=1}^{n_D} |\calR(\theta)_i|^2 + \lambda_B\sum_{i=n_D+1}^n |\calR(\theta)_i|^2,
\end{align*}
where we refer to $\lambda_D, \lambda_B > 0$ as our loss weights.

\begin{remark}
    It is common to refer to the network parameters $\theta$ as the weights of a neural network, and this should not be confused with the loss weights $\lambda_D,\lambda_B$ defined above. The way in which we use ``weights" and ``parameters" should be clear from the context, but we will never refer to the loss weights as parameters.
\end{remark}

More generally, one could minimize the loss function $\L:\R^p\to[0,\infty)$ given by
\begin{equation}
\label{eq:objective}
    \L(\theta) = \frac12 \calR(\theta)^\top\Lambda \calR(\theta),
\end{equation}
where $\Lambda \in \R^{n\times n}$ is a fixed symmetric and positive definite matrix. 

It follows that $\mathcal{L}$ is the squared norm of $\calR$ with respect to some inner product determined by $\Lambda$ (which is yet to be chosen), thus we will also use the notation $\mathcal{L}(\theta) = \frac12 \|\calR(\theta)\|_\Lambda^2$ to denote the same quantity. We remark that this generalization includes the case in which we give three separate weights for the interior PDE residual, the boundary conditions, and the initial condition for time dependent PDEs. In this case, we will denote the number of points used for different parts of the boundary of $\Omega$ as $n_{B_1}, n_{B_2},\dots$, and the number of points used for the initial condition and initial time derivative as $n_{B_i}$ and $n_{D_i}$, respectively. We will also differentiate between norms by using the notation $\|\cdot\|$ to denote the spectral norm for matrices and the Euclidean norm for vectors, as well as $\|\cdot\|_F$ for the Frobenius norm. Analogously, $\langle\cdot,\cdot\rangle$ will denote the Euclidean inner product and $\langle\cdot,\cdot\rangle_F$ the Frobenius inner product.

\subsection{Neural Tangent Kernel Weights for Loss}

Having defined the PINN loss, we now consider the gradient flow used to identify suitable parameters $\theta$ that make $u_\theta$ a good approximation to the PDE solution $u$. We start by assuming that $\Lambda$ is independent of $\theta$, as is the typical context for training PINNs with non-adaptive weights. The reason for starting with non-adaptive weights is to motivate the adaptive weighting algorithm from \cite{WangNTK2022}, which is the object of our analysis. The learning process is thus governed by 
\begin{equation}
\label{eq:inner-product-gd}
    \frac{d\theta}{dt} = -\nabla\mathcal{L}(\theta(t)) = -\nabla \calR(\theta) \Lambda \calR(\theta),
\end{equation}
where $\nabla \calR(\theta) = \begin{bmatrix} \nabla \calR_1(\theta) & \dots & \nabla \calR_n(\theta)\end{bmatrix} \in \R^{p\times n}$. In general, the variables with respect to which the gradient is taken will be clear from the context. When it requires specifying, we will denote it as, \textit{e.g.}, $\nabla_\theta$.

On the other hand, when computing the time derivative of the residual term, we have
\[
    \frac{d\calR(\theta(t))}{dt} = \nabla\calR(\theta)^\top \frac{d\theta}{dt} = -\nabla\calR(\theta)^\top\nabla\calR(\theta)\Lambda\calR(\theta) = -K(\theta)\Lambda\calR(\theta),
\]
where $K(\theta) = \nabla\calR(\theta)^\top\nabla\calR(\theta)$ is defined as the \textit{neural tangent kernel} of the PINN. To better understand the contribution of the interior and boundary components, we can partition $K(\theta)$ into four blocks $K_{DD} \in \R^{n_D\times n_D}$, $K_{BB} \in \R^{n_B\times n_B}$, $K_{DB} \in \R^{n_D\times n_B}$, and $K_{BD} \in \R^{n_B\times n_D}$ given the original arrangement of the residual $\calR(\theta)$ so that
\[
    K(\theta) = \begin{pmatrix}
        K_{DD}(\theta) & K_{DB}(\theta)\\
        K_{BD}(\theta) & K_{BB}(\theta)
    \end{pmatrix}.
\]
Because $K(\theta) = \nabla\calR(\theta)^\top\nabla\calR(\theta)$ is symmetric, we have that $K_{DD}$ and $K_{BB}$ are symmetric, and $K_{DB} = K_{BD}^\top$. If we consider the discrete approximation of \eqref{eq:inner-product-gd}, \textit{i.e.}, the gradient descent algorithm
\begin{equation}\label{eq:non-adaptive-gd}
    \theta_{t+1} = \theta_t - \eta \nabla\calR(\theta_t)\Lambda\calR(\theta_t),\quad \theta_0 = \theta(0),
\end{equation}
where $\eta>0$ is the learning rate, then the loss weighting scheme proposed in \cite{WangNTK2022} suggests modifying $\Lambda$ at each iteration to obtain parameter-dependent weights $\Lambda = \Lambda(\theta_t)$ which balance the two components of the loss given by the PDE residuals ($\calR(\theta)_i$ for $1\le i\le n_D$) and the boundary residuals ($\calR(\theta)_i$ for $n_D+1 \le i \le n$). This results in the following iteration scheme with parameter-dependent weights 
\begin{equation}\label{eq:adaptive-gd}
    \theta_{t+1} = \theta_t - \eta\nabla\calR(\theta_t) \Lambda(\theta_t)\calR(\theta_t),\quad \theta_0 = \theta(0).
\end{equation}
One particular choice of $\Lambda(\theta)$ proposed in \cite{WangNTK2022} is to consider a term-wise rescaling    
\begin{equation}\label{eq:wang-inner-product}
    \Lambda(\theta_t) = \begin{pmatrix}
        \lambda_D(\theta_t)I_{n_D} & 0\\
        0 & \lambda_B(\theta_t)I_{n_B}
    \end{pmatrix},\quad \text{where} \quad \lambda_D(\theta_t) = \frac{\Tr(K(\theta_t))}{\Tr(K_{DD}(\theta_t))},\quad \lambda_B(\theta_t) = \frac{\Tr(K(\theta_t))}{\Tr(K_{BB}(\theta_t))},
\end{equation}
and $I_m$ denotes the $m\times m$ identity matrix. In practice, such a choice suggests a combination of the weight-independent scheme \eqref{eq:non-adaptive-gd} and the weight-dependent one \eqref{eq:adaptive-gd}, mainly for efficiency reasons, since adaptive weights \eqref{eq:wang-inner-product} are slow to compute. From a practical point of view, a way to overcome this issue is to perform many gradient descent steps with constant $\Lambda$ before being updated with the current weights $\theta_t$. The procedure is summarized in Algorithm \ref{alg:adaptive-gd}.

\begin{algorithm}
\caption{Gradient Descent with Adaptive Loss Weights}\label{alg:adaptive-gd}
\begin{algorithmic}[1]
\Require Update frequency \texttt{n}, number of training steps $T$
\State Initialize $\theta_0 \in \calP$
\For{$t = 0,1,\dots,T$}
    \If{$t$ is divisible by \texttt{n}}
        \State Compute $K(\theta_t)$ and define $\Lambda = \Lambda(\theta_t)$ according to \eqref{eq:wang-inner-product}
    \EndIf
    \State $\theta_{t+1} = \theta_t - \eta\nabla\calR(\theta_t)\Lambda\calR(\theta_t)$
\EndFor
\end{algorithmic}
\end{algorithm}

The heuristic motivation for the particular choice of $\Lambda$ in \eqref{eq:wang-inner-product} is that when we consider $u_\theta$ to be in a class of feedforward neural networks of fixed depth, then in a suitable infinite-width limit of this network for particular PDEs, $K(\theta_t) = K(\theta_0)$ remains constant throughout training (see Theorems 4.3 and 4.4 in \cite{WangNTK2022}), giving the continuous time dynamics
\begin{equation}\label{eq:ntk-limit}
    \frac{d\calR(\theta(t))}{dt} = -K\Lambda\calR(\theta(t)).
\end{equation}

Thus, the weighting parameters in Equation \eqref{eq:wang-inner-product} can be thought of as a rough approximation to a scaled inverse of $K$ (e.g., if $K$ were approximately diagonal), trying to balance the rates at which the PDE residuals and the boundary residuals decrease. Indeed, if $\Lambda=\Tr(K)K^{-1}$ holds exactly, then Equation \eqref{eq:ntk-limit} yields the solutions $\calR(\theta(t))_i = e^{-\Tr(K)t}\calR(\theta(0))_i$ for all $i$, which means that the rate $e^{-\Tr(K)t}$ is the same for each residual $\calR(\theta)_i$.

However, practically speaking, we cannot work with infinite-width neural networks, so the exact limit will not hold, and $K(\theta_t)$ will change with $t$. This means that $\Lambda(\theta_t)$ needs to be updated, possibly frequently, and without affecting the speed of training too much.

In the following, we will not focus on the concerns of efficiently exploiting Equations \eqref{eq:non-adaptive-gd} and \eqref{eq:adaptive-gd}, and instead assume that the loss weights can be updated at each iteration. Toward this goal, we present convergence results for the heuristically better weighting scheme \eqref{eq:adaptive-gd} in which the weights $\Lambda(\theta_t)$ evolve over time rather than keeping them constant in certain intervals. As a corollary to our results, we do offer a convergence result with fewer assumptions which gives a rule for the frequency with which updates to the NTK-based weights should be made. We will then present a solution to the efficiency issue in Section \ref{sec:fast-NTK}.


\section{Convergence Analysis}\label{sec:convergence}

We first consider the convergence of Algorithm \ref{alg:adaptive-gd}, keeping the weights $\Lambda(\theta_t)$ general, and not specifically defined by \eqref{eq:wang-inner-product}. The main difficulty in the analysis lies in the fact that $\Lambda$ is not fixed and may depend on the neural tangent kernel $\nabla\calR(\theta_t)^\top\nabla\calR(\theta_t)$. 

As a corollary of our analysis, we are also able to obtain results for the more specific choice of the weights given by \eqref{eq:wang-inner-product}. To investigate the convergence properties, we will require a number of assumptions.
\begin{assumption}[Convergence Assumptions]\label{assumption:convergence} In the following we will assume: 
    
    \begin{enumerate}[ref={\theassumption.\arabic*}] 
        \item\label{assumption:convex-parameter-domain} The parameter domain $\calP\subseteq\R^p$ is open and convex.
        \item\label{assumption:lambda-bounds} For all $t\in\N\cup\{0\}$, $\Lambda(\theta_t)$ is symmetric positive definite with eigenvalues bounded as $0 < \ell_{\min} \le \lambda_{\min}(\Lambda(\theta_t)) \le \lambda_{\max}(\Lambda(\theta_t)) \le \ell_{\max} < \infty$.
        \item\label{assumption:K-bounds} For all $t\in\N\cup\{0\}$, $K(\theta_t)$ is symmetric positive definite with eigenvalues bounded as $0 < k_{\min} \le \lambda_{\min}(K(\theta_t)) \le \lambda_{\max}(K(\theta_t)) \le k_{\max} < \infty$. 
        \item\label{assumption:uniformly-lipschitz-gradient} Let $G(\theta; z) = \frac12\|\calR(\theta)\|_{\Lambda(z)}^2$, then $\nabla_\theta G(\theta; z)$ is uniformly $L$-Lipschitz for any $z$, i.e., there exists $L>0$ such that for any $\theta,\theta'\in\calP$ and for all $z\in\calP$, 
        \[
            \|\nabla_\theta G(\theta; z) - \nabla_\theta G(\theta'; z)\| \le L\|\theta - \theta'\|.
        \]
        \item\label{assumption:lipschitz-gradient} Let $F(\theta) = \frac12 \|\calR(\theta)\|^2$, then $\nabla F(\theta) = \nabla\calR(\theta)\calR(\theta)$ is $L$-Lipschitz. 
    \end{enumerate}
\end{assumption}
Assumptions \ref{assumption:convex-parameter-domain}, \ref{assumption:uniformly-lipschitz-gradient}, and \ref{assumption:lipschitz-gradient} are technical assumptions required to apply the standard descent lemma below \cite{WardDescentLemma, nesterov2004introductory,StochasticProgramming}. Moreover, Assumption \ref{assumption:convex-parameter-domain} is not very restrictive, since, for example, we could take $\calP = \R^p$ for our convergence results. We also remark that if the weights $\Lambda(\theta_t)$ are taken as in \cite{WangNTK2022}, then Assumption \ref{assumption:lambda-bounds} is automatically satisfied by Assumption \ref{assumption:K-bounds}. Furthermore, in the regime in which the NTK converges in the infinite-width limit, the upper bound in Assumption \ref{assumption:K-bounds} will be satisfied. 
The lower bound depend on the sample of training data. If, for instance, the gradients of the residuals for the training points have a strong linear correlation, then $\nabla\calR(\theta)^\top\nabla \calR(\theta)$ could be approximately low rank. Moreover, the lower bound cannot be satisfied if $n > p$ since the rank of $K(\theta) = \nabla\calR(\theta)^\top\nabla\calR(\theta)$ is bounded by $\operatorname{rank}\nabla\calR(\theta) \le \min\{p, n\}$, making $K(\theta)$ rank-deficient. In other words, there must be more parameters in the neural network than the amount of data used to train the model. 

\begin{remark}
    For certain linear PDEs, if one considers the infinite-width NTK limit, it may be possible to derive even stronger convergence results. We instead focus on a more general and realistic setting without assuming that our differential operator is linear and the neural networks have finite width.
\end{remark}

\begin{lemma}[Descent Lemma]\label{lem:descent-lemma}
    If $F \in C^1(\calP)$ with $\calP\subseteq\R^p$ convex and $\nabla F$ is $L$-Lipschitz, then for all $x,y\in\calP$,
    \[
        F(x) - F(y) \le \langle \nabla F(y), x-y\rangle + \frac{L}{2}\|x - y\|^2.
    \]
\end{lemma}
\begin{proof}
    See Appendix \ref{sec:proof-descent-lemma}.
\end{proof}

Using this result, we can prove that the time average of the residual norms $\|\calR(\theta_t)\|^2$ converges. This will be used later to prove that the loss gradients also converge in some sense.

\begin{theorem}[Convergence of Residual Averages]\label{thm:convergence-residuals}
    Let $\{\theta_t\}_t$ be computed via the scheme \eqref{eq:adaptive-gd}. Then, under Assumptions \ref{assumption:convex-parameter-domain}, \ref{assumption:lambda-bounds}, \ref{assumption:K-bounds}, and \ref{assumption:lipschitz-gradient}, it holds
    \[
        \frac{1}{T}\sum_{t=0}^{T-1} \|\calR(\theta_t)\|^2 \le \frac{\|\calR(\theta_0)\|^2 - \|\calR(\theta_T)\|^2}{T\eta}, \quad \text{for all}\quad  \eta \le \frac{2k_{\min}\ell_{\min}}{Lk_{\max}\ell_{\max}^2}, \quad \text{and} \quad T\ge 1.
    \]
    In particular, the following limit holds
    \[
        \lim_{T\to\infty}\frac{1}{T}\sum_{t=0}^{T-1} \|\calR(\theta_t)\|^2 = 0.
    \]
\end{theorem}
\begin{proof}
    By defining $F(\theta) = \frac12\|\calR(\theta)\|^2$, we note that $\nabla F(\theta) = \nabla\calR(\theta) \calR(\theta)$ is $L$-Lipschitz by Assumption \ref{assumption:lipschitz-gradient}. Then by Lemma \ref{lem:descent-lemma} combined with \eqref{eq:adaptive-gd}, we have
    \begin{equation}\label{eq:res-avg-desclem}
    \begin{split}
        F(\theta_{t+1}) &\le F(\theta_t) + \langle \nabla F(\theta_t), \theta_{t+1} - \theta_t\rangle + \frac{L}{2} \|\theta_{t+1} - \theta_t\|^2\\
        &= F(\theta_t) - \eta\calR(\theta_t)^\top\nabla\calR(\theta_t)^\top\nabla\calR(\theta_t)\Lambda(\theta_t)\calR(\theta_t) + \frac{L\eta^2}{2} \|\nabla\calR(\theta_t)\Lambda(\theta_t)\calR(\theta_t)\|^2.
    \end{split}
    \end{equation}
    Now, note that for symmetric positive definite matrices $A$ and $B$, the following bound holds
    \[
        \frac{1}{\lambda_{\min}(AB)} = \lambda_{\max}((AB)^{-1}) = \|(AB)^{-1}\| \le \|A^{-1}\|\|B^{-1}\| = \lambda_{\max}(A^{-1})\lambda_{\max}(B^{-1}) = \frac{1}{\lambda_{\min}(A)\lambda_{\min}(B)}.
    \]
    Taking $A = K(\theta_t)$ and $B = \Lambda(\theta_t)$, we have that $\lambda_{\min}(K(\theta_t)\Lambda(\theta_t)) \ge \lambda_{\min}(K(\theta_t))\lambda_{\min}(\Lambda(\theta_t))$, since $K(\theta_t) = \nabla\calR(\theta_t)^\top\nabla\calR(\theta_t)$ and $\Lambda(\theta_t)$ are symmetric positive definite matrices by Assumptions \ref{assumption:lambda-bounds} and \ref{assumption:K-bounds}. Combining this with the fact that $\operatorname{tr}(CD) \ge \lambda_{\min}(C)\operatorname{tr}(D)$ for symmetric positive definite matrices $C$ and $D$ and with the cyclic property of the trace, it follows that
    \begin{align*}
        \calR(\theta_t)^\top \nabla\calR(\theta_t)^\top\nabla\calR(\theta_t)\Lambda(\theta_t)\calR(\theta_t) 
        &= \operatorname{tr}(\calR(\theta_t)^\top \nabla\calR(\theta_t)^\top\nabla\calR(\theta_t)\Lambda(\theta_t)\calR(\theta_t))\\
        &= \|\calR(\theta_t)\|^2\operatorname{tr}\left(\underbrace{\nabla\calR(\theta_t)^\top\nabla\calR(\theta_t)\Lambda(\theta_t)}_{C}\underbrace{\frac{\calR(\theta_t)}{\|\calR(\theta_t)\|}\frac{\calR(\theta_t)^\top}{\|\calR(\theta_t)\|} }_{D}\right)\\
        &\ge \|\calR(\theta_t)\|^2\lambda_{\min}(\nabla\calR(\theta_t)^\top\nabla\calR(\theta_t)\Lambda(\theta_t))\operatorname{tr}\left(\frac{\calR(\theta_t)}{\|\calR(\theta_t)\|}\frac{\calR(\theta_t)^\top}{\|\calR(\theta_t)\|}\right)\\
        &\ge \lambda_{\min}(K(\theta_t))\lambda_{\min}(\Lambda(\theta_t))\|\calR(\theta_t)\|^2,
    \end{align*}
    where we used the fact that by the cyclic property of the trace, it holds $\operatorname{tr}\left(\frac{\calR(\theta_t)}{\|\calR(\theta_t)\|}\frac{\calR(\theta_t)^\top}{\|\calR(\theta_t)\|}\right) = 1$. Additionally, we have the inequality
    \[
        \|\nabla\calR(\theta_t)\Lambda(\theta_t)\calR(\theta_t)\|^2 = \calR(\theta_t)^\top \Lambda(\theta_t)\nabla\calR(\theta_t)^\top \nabla\calR(\theta_t)\Lambda(\theta_t)\calR(\theta_t) \le \lambda_{\max}(K(\theta_t)) \lambda_{\max}(\Lambda(\theta_t))^2\|\calR(\theta_t)\|^2.
    \]
    Using  these bounds in Equation \eqref{eq:res-avg-desclem} together with Assumptions \ref{assumption:lambda-bounds} and \ref{assumption:K-bounds} yields
    \begin{align*}
        F(\theta_{t+1}) &\le F(\theta_t) - \eta\lambda_{\min}(K(\theta_t))\lambda_{\min}(\Lambda(\theta_t))\|\calR(\theta_t)\|^2 + \frac{L\eta^2}{2} \lambda_{\max}(K(\theta_t)) \lambda_{\max}(\Lambda(\theta_t))^2\|\calR(\theta_t)\|^2\\
        &= F(\theta_t) - \left(\eta\lambda_{\min}(K(\theta_t))\lambda_{\min}(\Lambda(\theta_t)) - \frac{L\eta^2}{2} \lambda_{\max}(K(\theta_t)) \lambda_{\max}(\Lambda(\theta_t))^2\right)\|\calR(\theta_t)\|^2\\
        &\le F(\theta_t) - \left(\eta k_{\min}\ell_{\min} - \frac{L\eta^2}{2} k_{\max}\ell_{\max}^2\right)\|\calR(\theta_t)\|^2.
    \end{align*}
    Dividing by $T\eta$ and summing over $t=0,\dots,T-1$ gives
    \[
        \left(k_{\min}\ell_{\min} - \frac{L\eta}{2}k_{\max}\ell_{\max}^2\right)\frac{1}{T}\sum_{t=0}^{T-1} \|\calR(\theta_t)\|^2 \le \frac{F(\theta_0) - F(\theta_T)}{T\eta}.
    \]
    Since $\eta \le \frac{2k_{\min}\ell_{\min}}{Lk_{\max}\ell_{\max}^2}$, the result follows by definition of $F$.
\end{proof}

\begin{remark}
    Note that the admissible learning rate $\eta \le \frac{2k_{\min}\ell_{\min}}{Lk_{\max}\ell_{\max}^2}$ depends on the inverses of the worst possible condition numbers of $K$ and $\Lambda$. This means that, practically speaking, while the average squared residuals will converge to zero by Theorem \ref{thm:convergence-residuals}, the required learning rate may need to be very small.
\end{remark}

Using the previous result, we can show the convergence of the gradients of $G$ as the inner product changes.

\begin{theorem}[Convergence of Gradient Averages]\label{thm:convergence-gradients}
     Let $\{\theta_t\}_t$ be computed via the scheme \eqref{eq:adaptive-gd}. Then under Assumptions \ref{assumption:convex-parameter-domain}--\ref{assumption:lipschitz-gradient}, it holds that
     \[
        \frac{1}{T}\sum_{t=0}^{T-1}\|\nabla_\theta G(\theta_t;\theta_t)\|^2 \le \frac{G(\theta_0;\theta_0) - G(\theta_T;\theta_T)}{T\eta(1-L\eta/2)} + \frac{2\ell_{\max}}{T\eta^2}\|\calR(\theta_0)\|^2,
    \]
    for any $\eta < \min\left\{\frac{2k_{\min}\ell_{\min}}{Lk_{\max}\ell_{\max}^2}, \frac{2}{L}\right\}$ and $T\ge 1$. In particular,
    \[
        \min_{t=0,1,\dots,T-1} \|\nabla_\theta G(\theta_t; \theta_t)\|^2 \le \frac1T\sum_{t=0}^{T-1}\|\nabla_\theta G(\theta_t;\theta_t)\|^2 \to 0 \quad \text{as} \quad T\to\infty.
    \]
\end{theorem}
\begin{proof}
Having defined $G(\theta; z) = \frac12\|\calR(\theta)\|_{\Lambda(z)}^2$, we remark that Equation \eqref{eq:adaptive-gd} can be rewritten as
\begin{equation}\label{eq:inner-product-sgd-update-F-formulation}
    \theta_{t+1} = \theta_t - \eta \nabla_\theta G(\theta_t; \theta_t).    
\end{equation}
By Assumption \ref{assumption:uniformly-lipschitz-gradient} and Lemma \ref{lem:descent-lemma}, for all $z\in\calP$, we have
\[
    G(\theta_{t+1}; z) - G(\theta_t; z) \le \langle \nabla_\theta G(\theta_t; z), \theta_{t+1} - \theta_t\rangle + \frac{L}{2}\|\theta_{t+1} - \theta_t\|^2,
\]
that for the specific choice of $z = \theta_t$, and using Equation \eqref{eq:inner-product-sgd-update-F-formulation} gives
\begin{equation}\label{eq:G-inequality}
    G(\theta_{t+1}; \theta_t) - G(\theta_t; \theta_t) \le -\eta\| \nabla_\theta G(\theta_t; \theta_t)\|^2 + \frac{L\eta^2}{2}\|\nabla_\theta G(\theta_t;\theta_t)\|^2.
\end{equation}
On the other side, we note that 
\begin{align*}
    G(\theta_{t+1}; \theta_{t+1}) - G(\theta_{t+1}; \theta_t)
        &= \frac12\|\calR(\theta_{t+1})\|_{\Lambda(\theta_{t+1})}^2 - \frac12\|\calR(\theta_{t+1})\|_{\Lambda(\theta_{t})}^2\\
        &= \frac12 \calR(\theta_{t+1})^\top (\Lambda(\theta_{t+1}) - \Lambda(\theta_t))\calR(\theta_{t+1})\\
        &\le \frac12 \lambda_{\max}(\Lambda(\theta_{t+1}) - \Lambda(\theta_t)) \|\calR(\theta_{t+1})\|^2,
\end{align*}
so that, combining the former inequality with Equation \eqref{eq:G-inequality}, we have
\begin{equation*}
    G(\theta_{t+1};\theta_{t+1}) - G(\theta_t;\theta_t) \le \frac12 \lambda_{\max}(\Lambda(\theta_{t+1}) - \Lambda(\theta_t)) \|\calR(\theta_{t+1})\|^2 - \eta\| \nabla_\theta G(\theta_t; \theta_t)\|^2 + \frac{L\eta^2}{2}\|\nabla_\theta G(\theta_t;\theta_t)\|^2.
\end{equation*}
Similarly as before, dividing by $T\eta$ and summing over $t=0,1,\dots,T-1$ gives
\begin{equation}\label{eq:inner-product-sgd-bound-general}
    \left(1 - \frac{L\eta}{2}\right)\frac{1}{T}\sum_{t=0}^{T-1}\|\nabla_\theta G(\theta_t;\theta_t)\|^2 \le \frac{G(\theta_0;\theta_0) - G(\theta_T;\theta_T)}{T\eta} + \frac{1}{2T\eta}\sum_{t=0}^{T-1} \lambda_{\max}(\Lambda(\theta_{t+1}) - \Lambda(\theta_t))\|\calR(\theta_{t+1})\|^2.
\end{equation}
Now by Assumption \ref{assumption:lambda-bounds} and Theorem \ref{thm:convergence-residuals}, we have
\begin{align*}
    &\frac{1}{2T\eta}\sum_{t=0}^{T-1} \lambda_{\max}(\Lambda(\theta_{t+1}) - \Lambda(\theta_t))\|\calR(\theta_{t+1})\|^2\\ 
    &\le \frac{\ell_{\max}}{T\eta}\sum_{t=0}^{T-1} \|\calR(\theta_{t+1})\|^2 \le \frac{2\ell_{\max}}{\eta(T+1)}\sum_{t=0}^T \|\calR(\theta_t)\|^2 \le \frac{2\ell_{\max}}{T\eta^2}\|\calR(\theta_0)\|^2,
\end{align*}
from which it follows that
\[
    \frac{1}{T}\sum_{t=0}^{T-1}\|\nabla_\theta G(\theta_t;\theta_t)\|^2 \le \frac{G(\theta_0;\theta_0) - G(\theta_T;\theta_T)}{T\eta(1-L\eta/2)} + \frac{2\ell_{\max}}{T\eta^2}\|\calR(\theta_0)\|^2,
\]
as desired.
\end{proof}

Using this result, we can also show the convergence of the gradients of $F$ if the weights $\Lambda$ are chosen according to Equation \eqref{eq:wang-inner-product}, as done in \cite{WangNTK2022}.
\begin{corollary}[Convergence for NTK-Based Weights]
    Let $\{\theta_t\}_t$ be computed via the scheme \eqref{eq:adaptive-gd}, and assume that $\Lambda(\theta_t)$ is computed according to \eqref{eq:wang-inner-product}. Then under Assumption \ref{assumption:convergence}, it holds that
     \[
        \frac{1}{T}\sum_{t=0}^{T-1} \|\nabla F(\theta_t)\|^2 \le \frac{k_{\max}}{k_{\min}}\left(\frac{G(\theta_0;\theta_0) - G(\theta_T;\theta_T)}{T\eta(1-L\eta/2)} + \frac{2\ell_{\max}}{T\eta^2}\|\calR(\theta_0)\|^2\right),
    \]
    for any $\eta < \min\left\{\frac{2k_{\min}\ell_{\min}}{Lk_{\max}\ell_{\max}^2}, \frac{2}{L}\right\}$ and $T\ge 1$. In particular,
    \[
        \min_{t=0,1,\dots,T-1} \|\nabla F(\theta_t)\|^2 \le \frac1T\sum_{t=0}^{T-1}\|\nabla F(\theta_t)\|^2 \to 0  \quad \text{as} \quad T\to\infty.
    \]
\end{corollary}
\begin{proof}
    For all $t$, we have
    \[
        \|\nabla_\theta G(\theta_t; \theta_t)\|^2 = \|\nabla\calR(\theta_t)\Lambda(\theta_t)\calR(\theta_t)\|^2 \ge \lambda_{\min}(K(\theta_t))\|\Lambda(\theta_t)\calR(\theta_t)\|^2 \ge k_{\min}\|\calR(\theta_t)\|^2,
    \]
    and
    \[
        \|\nabla F(\theta_t)\|^2 = \|\nabla\calR(\theta_t)\calR(\theta_t)\|^2 \le \|\nabla\calR(\theta_t)\|^2 \|\calR(\theta_t)\|^2 \le k_{\max}\|\calR(\theta_t)\|^2
    \]
    by Assumption \ref{assumption:K-bounds}. Combining these gives $\|\nabla_\theta G(\theta_t;\theta_t)\|^2 \ge \frac{k_{\min}}{k_{\max}}\|\nabla F(\theta_t)\|^2$ so that by Theorem \ref{thm:convergence-gradients}, we obtain
    \[
        \frac{1}{T}\sum_{t=0}^{T-1} \|\nabla F(\theta_t)\|^2 \le \frac{k_{\max}}{k_{\min}}\left(\frac{G(\theta_0;\theta_0) - G(\theta_T;\theta_T)}{T\eta(1-L\eta/2)} + \frac{2\ell_{\max}}{T\eta^2}\|\calR(\theta_0)\|^2\right),
    \]
    as desired.
\end{proof}

We now present a corollary that provides results analogous to the ones in Theorem \ref{thm:convergence-gradients} but under fewer assumptions. In particular, we do not require assumptions on $\Lambda$ and $K$, as long as the inner product $\Lambda$ is changed only periodically according to some prescribed rule. Intuitively, if $\Lambda$ does not change between iterations, then the extra term in \eqref{eq:inner-product-sgd-bound-general} depending on the largest eigenvalue of the difference between $\Lambda$ at successive time steps is zero, in which case we would have the desired convergence. Thus, if the changes of $\Lambda$ are sufficiently sparse in time, then we can still obtain the convergence of Theorem \ref{thm:convergence-gradients} without invoking Theorem \ref{thm:convergence-residuals} in the proof.

\begin{corollary}[Controlled Gradients via Spaced Updates]
    Let $\{\theta_t\}_t$ be computed via the scheme \eqref{eq:adaptive-gd}, and let $h:[0,\infty)\to[0,\infty)$ be any non-decreasing function of $T$ such that $h = o(T)$ as $T\to\infty$. Letting $\widetilde \Lambda(\theta_{t+1})$ be any predicted update to $\Lambda(\theta_t)$ and $$S(t) = \begin{cases} \sum_{r=0}^{t-1} \lambda_{\max}(\Lambda(\theta_{r+1}) - \Lambda(\theta_r))\|\calR(\theta_{r+1})\|^2 & \text{if } t > 0,\\ 0 &\text{otherwise},\end{cases}$$ consider the rule
    \begin{equation}\label{eq:update-rule}
        \Lambda(\theta_{t+1}) = \begin{cases}
            \widetilde{\Lambda}(\theta_{t+1}) & \text{if } S(t) + \lambda_{\max}(\widetilde{\Lambda}(\theta_{t+1}) - \Lambda(\theta_t))\|\calR(\theta_{t+1})\|^2 \le h(t),\\
            \Lambda(\theta_t) &\text{otherwise}.
        \end{cases}
    \end{equation}
    Then, under Assumptions \ref{assumption:convex-parameter-domain}, \ref{assumption:uniformly-lipschitz-gradient} and assuming $\eta < \frac{2}{L}$, we have
    \[
        \min_{t=0,1,\dots,T-1} \|\nabla_\theta G(\theta_t; \theta_t)\|^2 \le \frac1T\sum_{t=0}^{T-1}\|\nabla_\theta G(\theta_t;\theta_t)\|^2 \to 0  \quad \text{as} \quad T\to\infty.
    \]
\end{corollary}
\begin{proof}
Starting from \eqref{eq:inner-product-sgd-bound-general}, applying the update rule \eqref{eq:update-rule}, and dividing by $1 - L\eta/2$, we have 
\[
    \frac1T\sum_{t=0}^{T-1}\|\nabla_\theta G(\theta_t;\theta_t)\|^2 \le \frac{G(\theta_0;\theta_0) - G(\theta_T;\theta_T)}{\eta T(1-L\eta/2)} + \frac{h(T)}{2\eta(1-L\eta/2) T}.
\]
Since $h=o(T)$ by assumption, the result follows.
\end{proof}

\begin{remark}
    The update rule \eqref{eq:update-rule} is based on a technical condition chosen specifically to control the extra term $\frac{1}{2T\eta}\sum_{t=0}^{T-1} \lambda_{\max}(\Lambda(\theta_{t+1}) - \Lambda(\theta_t))\|\calR(\theta_{t+1})\|^2$ in \eqref{eq:inner-product-sgd-bound-general}. If an additional term $\lambda_{\max}(\widetilde{\Lambda}(\theta_{t+1}) - \Lambda(\theta_t))\|\calR(\theta_{t+1})\|^2$ in the sum causes a growth rate $h(t) = o(t)$ to be exceeded, then the loss weights $\Lambda$ are not updated.
\end{remark}

\section{Fast Estimation of the Neural Tangent Kernel}\label{sec:fast-NTK}
To address the computational bottleneck of computing NTK-based loss weights in \eqref{eq:wang-inner-product}, we now present a practical algorithm to obtain fast approximations of the Neural Tangent Kernel $K(\theta_t)$ based on matrix sketching techniques \cite{RandomizedLinearAlgebraSurvey}. We first describe the algorithm and then provide intuition for why it works, showing a result on the estimator's unbiasedness (up to a discretization parameter). In particular, our objective is to find an approximation of $K(\theta_t)$ in expectation by means of $\widehat{K}(\theta_t)$.

To define a single sample approximation of $K(\theta_t)$, let us suppose that we have already computed $\theta_t$, and set $\Lambda_{\text{pred}}(\theta_t) = \diag(-g_t)\diag(\calR(\theta_t))^\dagger$, where $g_t \sim N(0, I_n)$, the notation $\diag(v)$ denotes the diagonal matrix with diagonal entries given by the entries of the vector $v$, and $A^\dagger$ denotes the Moore-Penrose pseudoinverse of the matrix $A$. We remark that, if none of the residual entries $\calR(\theta_t)_i$ is equal to 0, then $\diag(\calR(\theta_t))^\dagger = \diag(\calR(\theta_t))^{-1}$. The approximation we present here follows without assumptions on the residuals, while in Theorem \ref{thm:unbiased-estimates-up-to-discretization} we will later assume the residuals are nonzero to ensure that our approximation is indeed close to $K(\theta_t)g_t$. Thus, for now, let us define the following iteration scheme
\[
    \widehat{\theta}_{t+1} = \theta_t - \Delta t \nabla\calR(\theta_t)\Lambda_{\text{pred}}(\theta_t)\calR(\theta_t), \quad \text{for any}\quad \Delta t>0.
\]
This step can be thought of as the ``predictor" step of a predictor-corrector method to numerically solve an ordinary differential equation, except that $\widehat{\theta}_{t+1}$ is not related to $\theta_{t+1}$ which will be computed later. Thus, the sole purpose of $\widehat{\theta}_{t+1}$ is to obtain approximations to the NTK as
\[
    K(\theta_t)g_t \approx \frac{\calR(\widehat{\theta}_{t+1}) - \calR(\theta_t)}{\Delta t}.
\]
We will show in the proof of Theorem \ref{thm:unbiased-estimates-up-to-discretization} that, through a Taylor expansion, this is indeed an approximation of $K(\theta_t)g_t$. Then, given this approximate matrix-vector product, we may form a single sample approximation of the NTK as
\[
    \widehat{K}(\theta_t) = \frac{\widetilde{K}(\theta_t) + \widetilde{K}(\theta_t)^\top}{2}, \quad \text{where}\quad \widetilde{K}(\theta_t) = \begin{pmatrix}
        \frac{\calR(\widehat{\theta}_{t+1}) - \calR(\theta_t)}{\Delta t} (g_t)_1 & \dots & \frac{\calR(\widehat{\theta}_{t+1}) - \calR(\theta_t)}{\Delta t} (g_t)_n
    \end{pmatrix} \in \R^{n\times n},
\]
and we can estimate the trace of the NTK as
\[
    \Tr(\widehat{K}(\theta_t)) = g_t^\top \left(\frac{\calR(\widehat{\theta}_t) - \calR(\theta_t)}{\Delta t}\right).
\]
Note that if our approximation of $K(\theta_t)g_t$ is correct, then $(g_t)_iK(\theta_t)g_t$ in expectation equals the $i$-th column of $K(\theta_t)$, which shows that $\widehat{K}(\theta_t)$ would indeed be an approximation of $K(\theta_t)$. Multiple samples constructed in this way may be averaged together to obtain a more accurate estimate. The steps for the single sample approximation are summarized in Algorithm \ref{alg:single-sample-approx}. Note that the only step with significant computational cost is the computation of $\widehat{\theta}_{t+1}$, which is simply the cost of a single extra forward and backward pass through the neural network with the data batch without significant extra memory costs. Now we can use the result of this ``predictor" step to define $\widehat{\Lambda}(\theta_t)$ analogously to \eqref{eq:wang-inner-product}:
\begin{equation}\label{eq:analogous-wang-inner-product}
    \widehat{\Lambda}(\theta_t) = \begin{pmatrix}
        \widehat{\lambda}_D(\theta_t)I_{n_D} & 0\\
        0 & \widehat{\lambda}_B(\theta_t)I_{n_B}
    \end{pmatrix},\quad \text{where} \quad \widehat{\lambda}_D(\theta_t) = \frac{\Tr(\widehat{K}(\theta_t))}{\Tr(\widehat{K}_{DD}(\theta_t))},\quad \widehat{\lambda}_B(\theta_t) = \frac{\Tr(\widehat{K}(\theta_t))}{\Tr(\widehat{K}_{BB}(\theta_t))},
\end{equation}
Having defined $\widehat{\Lambda}(\theta_t)$, we can perform the ``corrector" step given by the typical gradient descent step \eqref{eq:adaptive-gd}.

\begin{algorithm}
\caption{Single sample approximation of $K(\theta_t)$}\label{alg:single-sample-approx}
\begin{algorithmic}
\Require Assumption \ref{assumption:ntk-approximation}, $\theta_t$ is given.
\Ensure $\widehat{K}(\theta_t) \approx K(\theta_t)$ in expectation
\State Sample $g_t \sim N(0, I_n)$
\State $\Lambda_{\text{pred}}(\theta_t) \gets \diag(-g_t)\diag(\calR(\theta_t))^{-1}$ 
\State $\widehat{\theta}_{t+1} \gets \theta_t - \Delta t \nabla\calR(\theta_t)\Lambda_{\text{pred}}(\theta_t)\calR(\theta_t)$
\State $\widetilde{K}(\theta_t) \gets \begin{pmatrix}
        \frac{\calR(\widehat{\theta}_{t+1}) - \calR(\theta_t)}{\Delta t} (g_t)_1 & \dots & \frac{\calR(\widehat{\theta}_{t+1}) - \calR(\theta_t)}{\Delta t} (g_t)_n
    \end{pmatrix}$
\State $\widehat{K}(\theta_t) \gets \frac{\widetilde{K}(\theta_t) + \widetilde{K}(\theta_t)^\top}{2}$
\end{algorithmic}
\end{algorithm}

\subsection{Error Analysis of Fast NTK Approximation}

We begin the analysis of the proposed Algorithm \ref{alg:single-sample-approx} by introducing the assumptions that we will require for this purpose.

\begin{assumption}[Fast NTK Approximation Assumptions]\label{assumption:ntk-approximation}
Assume the following hold:
    
    \begin{enumerate}[ref={\theassumption.\arabic*}] 
        \item\label{assumption:nonzero-residuals} For all $t=0,1,\dots,T$, the residual $\calR(\theta_t)_i \ne 0$ for $i=1,2,\dots,n$.
        \item\label{assumption:bounded-parameter-domain} The parameter domain $\calP\subseteq\R^p$ is open, convex, and bounded.
        \item\label{assumption:bounded-residual-derivatives} The residual $\calR$ is smooth and has bounded derivatives in $\calP$.
    \end{enumerate}
\end{assumption}
Note that Assumption \ref{assumption:bounded-residual-derivatives} will be guaranteed by using a smooth activation function for our PINN (which will also be required in general for allowing higher order differential operators in the loss) and by Assumption \ref{assumption:bounded-parameter-domain}. The Assumption \ref{assumption:nonzero-residuals} is reasonable since we expect that the residuals are never exactly minimized in a typical PINN optimization. The second Assumption \ref{assumption:bounded-parameter-domain} then is the most restrictive but is reasonable if the gradient descent iterates are converging. Now under this assumption, we will show that Algorithm \ref{alg:single-sample-approx} provides an accurate approximation of the NTK.

\begin{theorem}[Unbiased Estimates of NTK up to Discretization]\label{thm:unbiased-estimates-up-to-discretization}
    Under Assumption \ref{assumption:ntk-approximation}, there exists a constant $C>0$ such that for any $\Delta t>0$ it holds
    \[
        \E\left[\left\|\frac{\calR(\widehat{\theta}_{t+1}) - \calR(\theta_t)}{\Delta t} - K(\theta_t)g_t\right\|\right] \le Cn\Delta t,\quad \text{and} \quad \E\left[\left\|\frac{\calR(\widehat{\theta}_{t+1}) - \calR(\theta_t)}{\Delta t} - K(\theta_t)g_t\right\|^2\right] \le C^2 n(n+2)\Delta t^2.
    \]
    In particular, the following are satisfied
    \[
        \left\|\E\left[\widehat{K}(\theta_t)\right] - K(\theta_t)\right\|_F \le Cn(n+1)\Delta t,\quad\text{and}\quad  \left|\E\left[\Tr(\widehat{K}(\theta_t))\right] - \Tr(K(\theta_t))\right| \le Cn(n+1)\Delta t.
    \]
\end{theorem}
\begin{proof}
Thanks to the fact that derivatives of $\calR$ are smooth in $\theta$ and bounded by Assumption \ref{assumption:bounded-residual-derivatives}, by definition of $\widehat{\Lambda}(\theta_t)$ we can use $\widehat{\theta}_{t+1} = \theta_t - \Delta t\nabla\calR(\theta_t)\Lambda_{\text{pred}}(\theta_t)\calR(\theta_t) = \theta_t + \Delta t\nabla\calR(\theta_t)g_t$, and then perform a Taylor expansion to obtain
\begin{equation}\label{eq:difference-quotient-bound}
\begin{split}
    \left\|\frac{\calR(\widehat{\theta}_{t+1}) - \calR(\theta_t)}{\Delta t} - K(\theta_t)g_t\right\|
    &= \left\|\frac{\calR(\theta_t + \Delta t\nabla\calR(\theta_t)g_t) - \calR(\theta_t)}{\Delta t} - K(\theta_t)g_t\right\|\\ 
    &\le \left\|\frac{\calR(\theta_t) + \Delta t\nabla\calR(\theta_t)^\top \nabla \calR(\theta_t)g_t - \calR(\theta_t)}{\Delta t} - K(\theta_t)g_t\right\| + \frac{C_1\Delta t^2\|\nabla\calR(\theta_t)g_t\|^2}{\Delta t}\\ 
    &= C_1\Delta t\|\nabla\calR(\theta_t)g_t\|^2 \le C_2\Delta t\|g_t\|^2.
\end{split}
\end{equation}
Taking the expectation of this yields the first part of the theorem since $\E[\|g_t\|^2] = n$ and $\E[\|g_t\|^4] = n(n+2)$, for the first and second inequality, respectively.

Now defining
\[
    \mathcal{K} = \frac12\left[\begin{pmatrix} (g_t)_1K(\theta_t)g_t & \dots & (g_t)_nK(\theta_t)g_t\end{pmatrix} + \begin{pmatrix} (g_t)_1K(\theta_t)g_t & \dots & (g_t)_nK(\theta_t)g_t\end{pmatrix}^\top\right],
\]
and using the fact that $\E[\mathcal{K}] = K(\theta_t)$ since $\E[(g_t)_iK(\theta_t)g_t] = K(\theta_t)e_i$ with $e_i$ the $i$-th standard basis vector of $\R^n$, we then have
\begin{align*}
    \left\|\E\left[\widehat{K}(\theta_t)\right] - K(\theta_t)\right\|_F
    &= \left\|\E\left[\widehat{K}(\theta_t)\right] - \E [\mathcal{K}]\right\|_F\\
    &\le \E\left[\left\|\widehat{K}(\theta_t) - \mathcal{K}\right\|_F\right]\\
    &\le \sum_{i=1}^n \E\left[|(g_t)_i|\cdot\left\|\frac{\calR(\widehat{\theta}_t)-\calR(\theta_t)}{\Delta t} - K(\theta_t)g_t\right\|\right]\\
    &\le \sum_{i=1}^n \sqrt{\E[(g_t)_i^2]}\sqrt{\E\left[\left\|\frac{\calR(\widehat{\theta}_t)-\calR(\theta_t)}{\Delta t} - K(\theta_t)g_t\right\|^2\right]}\\
    &\le n\sqrt{C^2 n(n+2)\Delta t^2} \le Cn(n+1)\Delta t,
\end{align*}
where in the second to last line, we used the Cauchy-Schwarz inequality, while in the last line we used $n(n+2) < (n+1)^2$, the fact that $\E[(g_t)_i^2]=1$, and our previously derived bound on the mean squared error of the difference quotient.

Lastly, following a similar sequence of inequalities, we have
\begin{align*}
    \left|\E\left[\Tr(\widehat{K}(\theta_t))\right] - \Tr(K(\theta_t))\right|
    &= \left|\E\left[\Tr(\widehat{K}(\theta_t))\right] - \E\left[\Tr(\mathcal{K})\right]\right|\\
    &\le \E\left[\left|\Tr(\widehat{K}(\theta_t) - \mathcal{K})\right|\right]\\
    &\le \sum_{i=1}^n \E\left[\left|\widehat{K}(\theta_t)_{ii} - \mathcal{K}_{ii}\right|\right]\\
    &= \sum_{i=1}^n \E\left[\left|\frac{\calR(\widehat{\theta}_t)_i - \calR(\theta_t)_i}{\Delta t}(g_t)_{i} - (K(\theta_t)g_t)_i(g_t)_i\right|\right]\\
    &\le \sum_{i=1}^n \sqrt{\E[(g_t)_i^2]}\sqrt{\E\left[\left|\frac{\calR(\widehat{\theta}_t)_i - \calR(\theta_t)_i}{\Delta t} - (K(\theta_t)g_t)_i\right|^2\right]}\\
    &\le \sum_{i=1}^n \sqrt{\E\left[\left\|\frac{\calR(\widehat{\theta}_t) - \calR(\theta_t)}{\Delta t} - K(\theta_t)g_t\right\|^2\right]}\\
    &\le Cn(n+1)\Delta t,
\end{align*}
as desired.
\end{proof}
\begin{remark}
    In place of Gaussian random vectors, there are other choices of vectors that could be used in Algorithm \ref{alg:single-sample-approx}. For example, one could choose Rademacher random vectors with i.i.d.\ entries taking the values $\pm1$, each with probability $1/2$. The main trick of the algorithm is that we can obtain a matrix-vector product of $K$ with any vector $v$ up to a discretization error of $O(\Delta t)$ by Taylor's theorem. 
\end{remark}

\begin{remark}
    By definition of the construction of $\widehat{K}(\theta_t)$, some of the entries of the estimator may be negative. It follows that $\widehat{K}(\theta_t)$ is dominated by an estimator obtained by taking the maximum of the entries of $\widehat{K}(\theta_t)$ and 0. In other words, one could define the estimator $\check{K}(\theta_t)$ given by
    \[
        \check{K}(\theta_t)_{ij} = \max(\widehat{K}(\theta_t)_{ij}, 0),\quad \text{for } i,j=1,\dots,n,
    \]
    which has a better mean absolute (and mean squared) error:
    \[
        \E\left[\|\check{K}(\theta_t) - K(\theta_t)\|_F\right] \le \E\left[\|\widehat{K}(\theta_t) - K(\theta_t)\|_F\right].
    \]
\end{remark}
We now present a result which provides a Monte Carlo error rate, up to a discretization error, for an approximation of the NTK obtained by averaging samples from Algorithm \ref{alg:single-sample-approx}.

\begin{theorem}[Error Rates of NTK Approximation]\label{thm:Monte-Carlo-error}
    Assume that Assumption \ref{assumption:ntk-approximation} holds, and let $\{\widehat{K}^{(j)}(\theta_t)\}_{j=1}^N$ be the $N$ results of Algorithm \ref{alg:single-sample-approx} with $N$ i.i.d.\ copies $g_t^{(j)}\sim N(0,I_n)$. Then for any $\Delta t>0$ it holds
    \[
        \E\left[\left|\frac{1}{N}\sum_{j=1}^N \Tr(\widehat{K}^{(j)}(\theta_t)) - \Tr(K(\theta_t))\right|^2\right] \le C n(n+2)(n+4)\Delta t^2 + \frac{4}{N}\|K(\theta_t)\|_F^2,
    \]
    and
    \[
        \E\left[\left\|\frac{1}{N}\sum_{j=1}^N \widehat{K}^{(j)}(\theta_t) - K(\theta_t)\right\|_F^2\right] \le Cn^2(n+2)(n+4)\Delta t^2 + \frac{2n(n+2)}{N}\|K(\theta_t)\|_F^2,
    \]
    where $C>0$ is a constant.
\end{theorem}
\begin{proof}
First, by the triangle inequality and Young's inequality we have
\begin{align*}
    &\E\left[\left|\frac{1}{N}\sum_{j=1}^N \Tr(\widehat{K}^{(j)}(\theta_t)) - \Tr(K(\theta_t))\right|^2\right] \\
    &\le \E\left[\underbrace{2\left|\frac{1}{N}\sum_{j=1}^N \Tr(\widehat{K}^{(j)}(\theta_t)) - \frac{1}{N}\sum_{j=1}^N (g_t^{(j)})^\top K(\theta_t)g_t^{(j)}\right|^2}_{(a)} + \underbrace{2\left|\frac{1}{N}\sum_{j=1}^N (g_t^{(j)})^\top K(\theta_t)g_t^{(j)} - \Tr(K(\theta_t))\right|^2}_{(b)}\right].
\end{align*}
For $(a)$, first note that for all $j=1,2,\dots,N$, by Taylor expansion as in the proof of Theorem \ref{thm:unbiased-estimates-up-to-discretization} we have,
\[
    \Tr(\widehat{K}^{(j)}(\theta_t)) = (g_t^{(j)})^\top\frac{\calR(\widehat{\theta}_t^{(j)}) - \calR(\theta_t)}{\Delta t} = (g_t^{(j)})^\top(K(\theta_t)g_t^{(j)} + \mathcal{E}_t^{(j)}),
\]
where $\|\mathcal{E}_t^{(j)}\| \le C\Delta t\|g_t^{(j)}\|^2$. It follows that
\begin{align*}
    (a) &= 2\E\left[\left|\frac{1}{N}\sum_{i=1}^N \left(\Tr(\widehat{K}^{(i)}(\theta_t)) - (g_t^{(i)})^\top K(\theta_t)g_t^{(i)}\right)\right|^2\right] = 2\E\left[\left|\frac{1}{N}\sum_{j=1}^N (g_t^{(j)})^\top\mathcal{E}_t^{(j)}\right|^2\right]\\
    &\le 2\E\left[\left(\frac{1}{N}\sum_{j=1}^N C\Delta t \|g_t^{(j)}\|^3\right)^2\right] \le \frac{2}{N^2}\E\left[N\sum_{j=1}^N C^2\Delta t^2\|g_t^{(j)}\|^6\right] = 2C^2\Delta t^2\E[\|g_t^{(1)}\|^6],
\end{align*}
where the last inequality comes from the fact that $(a_1+\dots+a_N)^2 \le N(a_1^2 + \dots + a_N^2)$. Using $\E[\|g_t^{(1)}\|^6] = n(n+2)(n+4)$, we obtain $(a) \le 2C^2n(n+2)(n+4)\Delta t^2$.

Now $(b)$ is the standard error for the Hutchinson estimator exploited in randomized numerical linear algebra to estimate the trace of a matrix using random matrix-vector products. By Theorem 2.1 in \cite{Girard1989}, we have that
\[
    (b) = 2\E\left[\left|\frac{1}{N}\sum_{j=1}^N (g_t^{(j)})^\top K(\theta_t)g_t^{(j)} - \Tr(K(\theta_t))\right|^2\right] = \frac{4}{N}\|K(\theta_t)\|_F^2.
\]
Combining the estimates for $(a)$ and $(b)$ yields the result.

For the second result, we have
\begin{align*}
    &\E\left[\left\|\frac{1}{N}\sum_{j=1}^N \widehat{K}^{(j)}(\theta_t) - K(\theta_t)\right\|_F^2\right]\\
    &\le 2\E\left[\left\|\frac{1}{N}\sum_{j=1}^N \widehat{K}^{(j)}(\theta_t) - \frac{1}{N}\sum_{j=1}^N \mathcal{K}^{(j)}(\theta_t)\right\|_F^2 + \left\|\frac{1}{N}\sum_{j=1}^N \mathcal{K}^{(j)}(\theta_t) - K(\theta_t)\right\|_F^2\right] =: (c) + (d).
\end{align*}

Considering $(c)$, we have 
\begin{align*}
    (c) &= 2\E\left[\left\|\frac{1}{N}\sum_{j=1}^N \widehat{K}^{(j)}(\theta_t) - \frac{1}{N}\sum_{j=1}^N \mathcal{K}^{(j)}(\theta_t)\right\|_F^2\right]\\
    &\le \frac{2}{N^2}\E\Bigg[\Bigg\|\sum_{j=1}^N \left(\begin{pmatrix}
        (g_t^{(j)})_1\frac{\calR(\widehat{\theta}_t^{(j)}) - \calR(\theta_t)}{\Delta t} & \dots & (g_t^{(j)})_n\frac{\calR(\widehat{\theta}_t^{(j)}) - \calR(\theta_t)}{\Delta t}
    \end{pmatrix}\right.\\ 
    &\hspace{10ex}\left.- \begin{pmatrix}
        (g_t^{(j)})_1K(\theta_t)g_t^{(j)} & \dots & (g_t^{(j)})_nK(\theta_t)g_t^{(j)}
    \end{pmatrix}\right)\Bigg\|_F^2\Bigg]\\
    &= \frac{2}{N^2} \E\left[\sum_{i=1}^N \left\|\sum_{j=1}^N (g_t^{(j)})_i\left(\frac{\calR(\widehat{\theta}_t^{(j)}) - \calR(\theta_t)}{\Delta t} - K(\theta_t)g_t^{(j)}\right)\right\|_F^2\right]\\
    &= \frac{2n}{N^2} \E\left[\left\|\sum_{j=1}^N (g_t^{(j)})_1\left(\frac{\calR(\widehat{\theta}_t^{(j)}) - \calR(\theta_t)}{\Delta t} - K(\theta_t)g_t^{(j)}\right)\right\|_F^2\right]\\
    &\le \frac{2n}{N} \E\left[\sum_{j=1}^N\left\|(g_t^{(j)})_1\left(\frac{\calR(\widehat{\theta}_t^{(j)}) - \calR(\theta_t)}{\Delta t} - K(\theta_t)g_t^{(j)}\right)\right\|_F^2\right]\\
    &\le Cn\Delta t^2\E[\|g_t^{(1)}\|^6] = Cn^2(n+2)(n+4)\Delta t^2,
\end{align*}
where in the last line, we applied Equation~\ref{eq:difference-quotient-bound}.

Now considering $(d)$, repeatedly using the fact that $\mathcal{K}^{(j)}(\theta_t)$ are i.i.d.\ with mean $K(\theta_t)$, we have
\begin{align*}
    (d) &= 2\E\left[\left\|\frac{1}{N}\sum_{j=1}^N \mathcal{K}^{(j)}(\theta_t) - K(\theta_t)\right\|_F^2\right]\\
    &= 2\E\left[\left\|\frac{1}{N}\sum_{j=1}^N \mathcal{K}^{(j)}(\theta_t)\right\|_F^2 - 2\left\langle \frac{1}{N}\sum_{j=1}^N \mathcal{K}^{(j)}(\theta_t), K(\theta_t)\right\rangle_F + \|K(\theta_t)\|_F^2\right]\\
    &= -2\|K(\theta_t)\|_F^2 + 2\E\left[\left\|\frac{1}{N}\sum_{j=1}^N \mathcal{K}^{(j)}(\theta_t)\right\|_F^2\right]\\
    &= -2\|K(\theta_t)\|_F^2 + \frac{2}{N^2}\sum_{j=1}^N \E\left[\left\|\mathcal{K}^{(j)}(\theta_t)\right\|_F^2\right] + \frac{2}{N^2}\sum_{i\ne j} \left\langle\E\left[\mathcal{K}^{(i)}(\theta_t)\right],\E\left[\mathcal{K}^{(j)}(\theta_t)\right]\right\rangle_F\\
    &= -2\|K(\theta_t)\|_F^2 + \frac{2}{N^2}\sum_{j=1}^N \E\left[\|K(\theta_t)+\mathcal{E}^{(j)}_t\|_F^2\right] + \frac{2(N-1)}{N}\|K(\theta_t)\|_F^2,
\end{align*}
where we redefine $\mathcal{E}_t^{(j)} = \mathcal{K}^{(j)}(\theta_t) - K(\theta_t)$. Now note that
\begin{align*}
    \frac{2}{N^2}\sum_{j=1}^N \E\left[\|K(\theta_t)+\mathcal{E}^{(j)}_t\|_F^2\right] &= \frac{2}{N}\left( \|K(\theta_t)\|_F^2 + \left\langle K(\theta_t), \E\left[\mathcal{E}_t^{(1)}\right]\right\rangle_F + \E\left[\|\mathcal{E}_t^{(1)}\|_F^2\right]\right)\\ 
    &= \frac{2}{N}\left( \|K(\theta_t)\|_F^2 + \E\left[\|\mathcal{E}_t^{(1)}\|_F^2\right]\right),
\end{align*}
where we used the fact that $\E[\mathcal{E}_t^{(1)}] = 0$. Now using that $\E[\|g_t\|^4] = n(n+2)$, we have
\begin{align*}
    \E\left[\|\mathcal{E}_t^{(1)}\|_F^2\right]
    &= \E[\|\mathcal{K}^{(1)}(\theta_t)\|_F^2] - \|K(\theta_t)\|_F^2\\
    &\le \E[\left\|\begin{pmatrix}
        (g_t)_1 K(\theta_t)g_t & \dots & (g_t)_n K(\theta_t)g_t
    \end{pmatrix}\right\|_F^2] - \|K(\theta_t)\|_F^2\\
    &= \E\left[\sum_{i=1}^n |(g_t)_i|^2\|K(\theta_t)g_t\|_F^2\right] - \|K(\theta_t)\|_F^2\\
    &\le \E[\|g_t\|^4]\cdot\|K(\theta_t)\|_F^2 - \|K(\theta_t)\|_F^2\\
    &\le n(n+2)\|K(\theta_t)\|_F^2.
\end{align*}
It follows that
\begin{align*}
    (d) &= \frac{2}{N}\E\left[\|\mathcal{E}_t^{(1)}\|_F^2\right] \le \frac{2n(n+2)}{N}\|K(\theta_t)\|_F^2.
\end{align*}
Combining this with the bound for $(c)$ gives the desired result.
\end{proof}

We now report another result on approximating the trace of the neural tangent kernel based on Theorem 4.1 in \cite{OtherTraceNTKResult} to compare it with our own strategy.
\begin{theorem}[Alternative Approximation of the NTK Trace]\label{thm:alternative-ntk-approximation}
    Let $g_t\sim N(0, \varepsilon^2I_n)$ and let $\theta_t$ be given. Under Assumption \ref{assumption:ntk-approximation}, there is $C>0$ such that
    \[
        - C(n+1)^2\varepsilon \le \E\left[\left\|\frac{\calR(\theta_t + g_t) - \calR(\theta_t)}{\varepsilon}\right\|_F^2\right] - \Tr(K(\theta_t)) \le C(n+1)^2(\varepsilon+\varepsilon^2).
    \]
\end{theorem}
\begin{proof}
    See Appendix \ref{sec:proof-alternative-ntk-approximation}.
\end{proof}
From Theorem \ref{thm:alternative-ntk-approximation} we can see that
\begin{equation}\label{eq:alternative-ntk-approximation}
    \frac{\calR(\theta_t + g_t) - \calR(\theta_t)}{\varepsilon},\quad\text{with } g_t\sim N(0,\varepsilon^2 I_n)
\end{equation}
also provides an estimate of the trace of $K(\theta_t)$, but since it does not allow arbitrary matrix-vector products with the NTK one cannot recover a full approximation of $K(\theta_t)$. Indeed, it may be seen from the proof of Theorem \ref{thm:alternative-ntk-approximation} given in Appendix \ref{sec:proof-alternative-ntk-approximation} that one can use the standard basis vectors of $\R^n$ to obtain approximations to the diagonal entries of $K(\theta_t)$, but not other entries of the NTK. Thus, the main difference between our estimator $\widehat{K}(\theta_t)$ and \eqref{eq:alternative-ntk-approximation} is that our approach gives additional information about off-diagonal entries of $K(\theta_t)$. In terms of computational cost, the proposed estimator requires a single extra backpropagation to compute $\widehat{\theta}_t$ compared to this alternative estimator. 

\subsection{Practical Algorithm with Moving Average}

We conclude this section by presenting a practical algorithm for computing an adaptive approximation of the NTK (or its trace). We begin by noting that an accurate computation of the NTK would take many samples due to the Monte Carlo error rate noted in Theorem \ref{thm:Monte-Carlo-error}. However, doing this every iteration would result in a high computational cost, which is exactly what we wish to avoid. Instead, we propose using a moving average over training steps of our single sample NTK estimates to obtain adaptive NTK estimates. This allows for an adaptive estimate of the NTK which costs just two extra forward passes and one extra backward pass in the neural network optimization per iteration, effectively only multiplying the training time by a factor of 2.5. Indeed, computing the weights $\widehat{\theta}_{t+1}$ is the same cost of a typical gradient descent step of a single forward and backward pass, and computing $\calR(\widehat{\theta}_{t+1})$ for the finite difference approximation is an extra forward pass ($\calR(\theta_t)$ need not be recomputed). These are the main costs since they depend on the number of network parameters which is thought to be large. Now, we can first obtain a good approximation of the NTK at the initialization:
\[
    \widehat{K}_0(\theta_0) = \frac{1}{N}\sum_{j=1}^N \widehat{K}^{(j)}(\theta_0),
\]
where we denote our moving average NTK approximation at time $t$ and parameter $\theta_t$ by $\widehat{K}_t(\theta_t)$ to differentiate it from the $N$ i.i.d.\ samples $\widehat{K}^{(j)}(\theta_t)$ taken according to Algorithm \ref{alg:single-sample-approx}. Then, for a fixed parameter $\alpha\in(0,1]$ and $t>0$, we compute
\[
    \widehat{K}_t(\theta_t) = (1-\alpha)\widehat{K}_{t-1}(\theta_{t-1}) + \alpha\widehat{K}^{(1)}(\theta_t),
\]
that is, for each new training step, we compute a single new sample from Algorithm \ref{alg:single-sample-approx} and average it with the NTK approximation from the previous time step. If the NTK were constant between iterations, the moving average would effectively become a sort of Monte Carlo estimator of the NTK, approximating the NTK better as $t\to\infty$. Thus, if the NTK changes smoothly between iterations, one would expect that the moving average will adapt to well-approximate the NTK over time. The full training procedure is summarized in Algorithm \ref{alg:moving-average-approx}, and an analogous procedure may be done for the trace of the NTK.

Lastly, note that as the learning rate $\eta\to 0$, the gradient descent updates change less and less between iterations due to the fact that the loss is smooth, so the moving average should become closer to the actual NTK.

\begin{algorithm}
\caption{Moving average approximation of $K(\theta_t)$}\label{alg:moving-average-approx}
\begin{algorithmic}
\Require Assumption \ref{assumption:ntk-approximation}, $\alpha\in(0,1]$ and $N$ are given.
\Ensure $\widehat{K}_t(\theta_t) \approx K(\theta_t)$ in expectation for all $t$.
\State Initialize $\theta_0$
\State $\widehat{K}_0(\theta_0) \gets \frac{1}{N}\sum_{j=1}^N \widehat{K}^{(j)}(\theta_0)$, \Comment{$\{\widehat{K}^{(j)}(\theta_0)\}_j$ are sampled i.i.d.\ according to Algorithm \ref{alg:single-sample-approx}}
\State $\widehat{\Lambda}(\theta_0) = \left[\frac{\Tr(\widehat{K}(\theta_0))}{\Tr({\widehat{K}}_{DD}(\theta_0))}I_{n_D}, 0;\ 0,\frac{\Tr(\widehat{K}(\theta_0))}{\Tr(\widehat{K}_{BB}(\theta_0))}I_{n_B}\right]$ \Comment{As in Eq.~\eqref{eq:wang-inner-product} with $\widehat{K}_0(\theta_0)$ in place of $K(\theta_0)$}
\State $\theta_1 \gets \theta_0 - \eta\nabla\calR(\theta_0)\widehat{\Lambda}(\theta_0)\calR(\theta_0)$
\For{$t=1,2,\dots,T$}
    \State $\widehat{K}_t(\theta_t) \gets (1-\alpha)\widehat{K}_{t-1}(\theta_{t-1}) + \alpha\widehat{K}^{(1)}(\theta_t)$, \Comment{$\widehat{K}^{(1)}(\theta_t)$ is sampled according to Algorithm \ref{alg:single-sample-approx}}
    \State $\widehat{\Lambda}(\theta_t) = \left[\frac{\Tr(\widehat{K}(\theta_t))}{\Tr({\widehat{K}}_{DD}(\theta_t))}I_{n_D}, 0;\ 0,\frac{\Tr(\widehat{K}(\theta_t))}{\Tr(\widehat{K}_{BB}(\theta_t))}I_{n_B}\right]$  \Comment{As in Eq.~\eqref{eq:wang-inner-product} with $\widehat{K}_t(\theta_t)$ in place of $K(\theta_t)$}  
    \State $\theta_{t+1} \gets \theta_t - \eta\nabla\calR(\theta_t)\widehat{\Lambda}(\theta_t)\calR(\theta_t)$
\EndFor
\end{algorithmic}
\end{algorithm}

\section{Numerical Results} \label{sec:numerics}

We now present numerical results to verify our theoretical findings and demonstrate the efficacy of the proposed method for fast NTK estimation. First, in Section \ref{ne:optimization-convergence} we present a one-dimensional example exhibiting the optimization convergence rates given in Theorems \ref{thm:convergence-residuals} and \ref{thm:convergence-gradients}. Then, we present the results for fast NTK estimation in Section \ref{ne:fast-ntk-estimation}, including: (i) a quadratically parameterized predictor (Section \ref{ne:quadratically-parameterized-predictor}), comparing the results to the exact NTK and verifying the Monte Carlo error rates in Theorem \ref{thm:Monte-Carlo-error}, (ii) a PINN problem to solve a wave equation (Section \ref{ne:wave-equation-pinn}), and (iii) a multidimensional nonlinear PDE modeling liquid crystals (Section \ref{ne:q-tensor}) \cite{gudibanda-weber-yue,hirsch-weber-yue}. Numerical examples involving PINNs have been developed within the framework of the open source software PINA \cite{pina}.

\subsection{Convergence Test}\label{ne:optimization-convergence}
In this experiment, we consider the second-order and one-dimensional equation
\[
    \begin{cases}
        u_{xx}(x) = -16\pi^2\sin(4\pi x), &\quad x\in[0,1],\\
        u(0) = u(1) = 0,
    \end{cases}
\]
and we seek a PINN $u_\theta$, parameterized by $\theta$, to approximate the solution to the equation, which is exactly given by $u(x) = \sin(4\pi x)$.

For this task, we use a neural network with one hidden layer and 100 neurons, hyperbolic tangent as the activation function, initialized with Xavier strategy, and with normalized input to have zero mean and unit standard deviation. To verify the theoretical results from Theorems \ref{thm:convergence-residuals} and \ref{thm:convergence-gradients}, we will use the exact NTK in the weights given by \eqref{eq:wang-inner-product}, following Algorithm \ref{alg:adaptive-gd} with $\texttt{n}=1$. Because computing the exact NTK is costly, we use a minimal number of collocation points to verify these results, taking $n_{B_1} = n_{B_2} = 1$ collocation points at the boundary points $x=0$ and $x=1$, as well as $n_D = 2$ collocation points in $(0, 1)$. With this number of points, we see that $n = n_D + n_{B_1} + n_{B_2} \le p$, which we recall was required for the lower bound in Assumption \ref{assumption:K-bounds} to hold.
The loss from the training is shown in Figure \ref{fig:convergence-experiment-collocation-pts-loss}.

\begin{figure*}[!htbp]
    \centering
    \includegraphics[width=.48\linewidth]{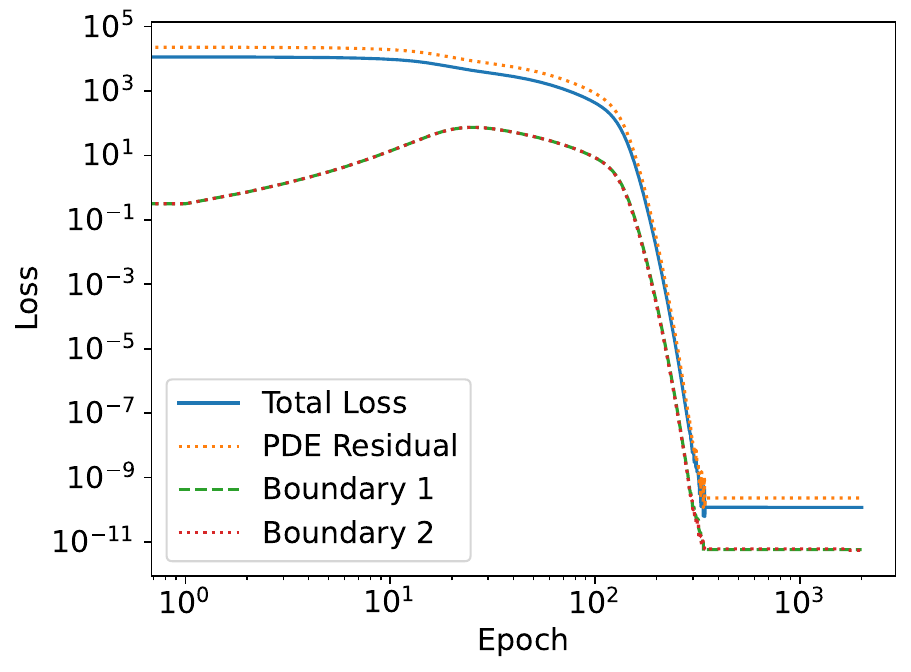}
    \caption{Convergence behavior for the total training loss, composed by equal contribution from the mean PDE residual, and the mean boundary terms.}
    \label{fig:convergence-experiment-collocation-pts-loss}
\end{figure*}

The plot in Figure \ref{fig:convergence-experiment-collocation-pts-loss} shows that the optimization task quickly minimized the loss up to machine precision. Because we take so few collocation points, we cannot expect the resulting PINN $u_\theta$ to approximate well the exact solution $u$, but this setting is already sufficient to see the optimization behavior described by our theory. Indeed, we start by verifying Assumptions \ref{assumption:lambda-bounds} and \ref{assumption:K-bounds} with the results in Figure \ref{fig:convergence-experiment-assumptions-verified}. In particular, the loss weights $\lambda_D$ and $\lambda_B$ remain bounded over the epochs as required by Assumption \ref{assumption:lambda-bounds}, and they remain stable at the end of training. Additionally, we observe that the NTK eigenvalues always increase during training, while they are no longer changing towards the end of optimization. The plot also shows a seemingly continuous evolution of the eigenvalues themselves. This is expected since the NTK is a continuous function of the neural network parameters which should also change continuously. We depicted the sorted eigenvalues of the NTK, so the colors of the eigenvalue trajectories had to be inferred by minimizing measures of distance between values and derivatives of the trajectories which we parameterized to produce what looked qualitatively correct. For example, the smallest eigenvalue at the initial time appears to evolve smoothly to become the largest eigenvalue at the final time, as shown in the blue curve.

\begin{figure*}[!htbp]
        \subfloat[Loss Weights]{%
            \includegraphics[width=.45\linewidth]{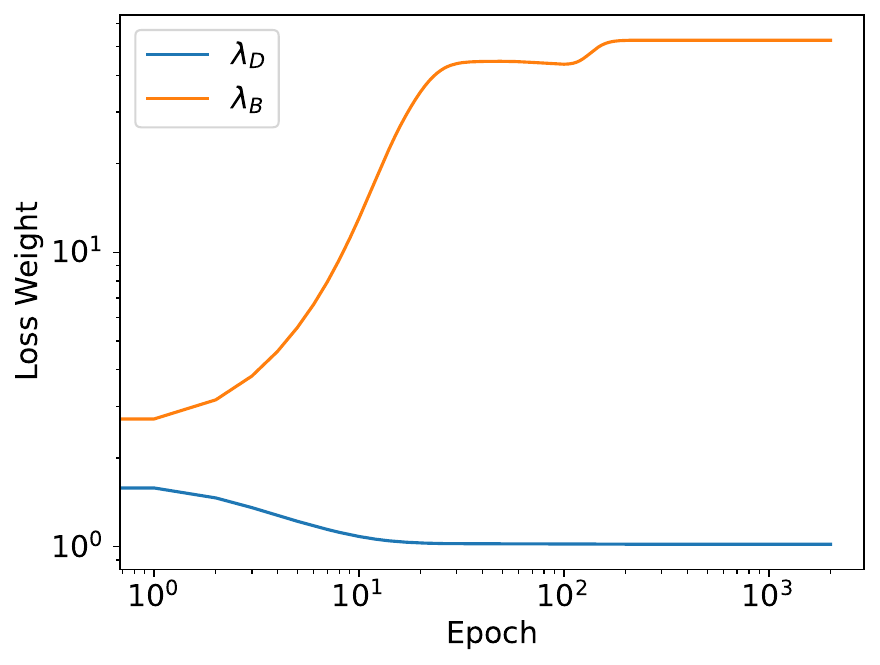}%
        }\hfill
        \subfloat[NTK Eigenvalues]{%
            \includegraphics[width=.45\linewidth]{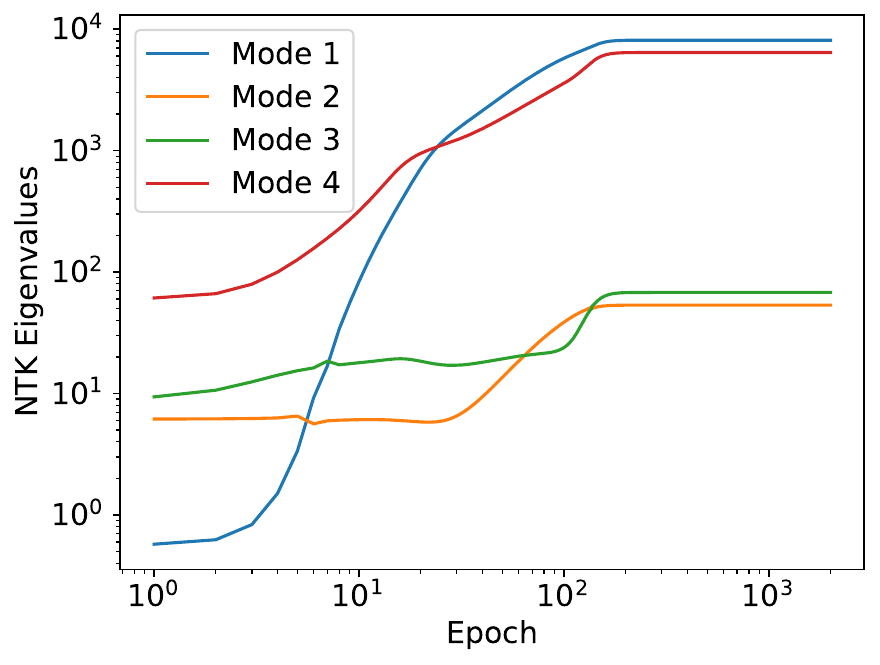}%
        }
        \caption{Loss weights and NTK eigenvalues which verify Assumptions \ref{assumption:lambda-bounds} and \ref{assumption:K-bounds}.}
        \label{fig:convergence-experiment-assumptions-verified}
\end{figure*}

To verify the predicted behavior obtained from Theorems \ref{thm:convergence-residuals} and \ref{thm:convergence-gradients}, we depict in Figure \ref{fig:convergence-experiment-time-averaged-quantities} the time averaged loss and time averaged squared norm of the gradient if the residuals, respectively. We can see that soon after the initialization the rates given by the theorems are empirically recovered during training. However, it is interesting to note that this rate is only observed after the losses appear to begin to decrease significantly (see Figure \ref{fig:convergence-experiment-collocation-pts-loss}). Indeed, the loss components appear to be more effectively minimized around epoch 100, when at the same time the NTK eigenvalues stabilize, eventually reaching machine precision at about epoch 200. These results correspond to the expected decay in Figure \ref{fig:convergence-experiment-time-averaged-quantities}, where around epoch 100 we begin to observe the $O(T^{-1})$ rate for the time averaged quantities obtained from our theoretical results. In summary, it seems that the NTK eigenvalues were increasing without bound until epoch 100, at which point the eigenvalues and gradient directions could stabilize. Then, the eigenvalues remained bounded, and our theory could be applied. However, the loss also exhibited a decay faster than $O(T^{-1})$, so this may indicate that the loss landscape was approximately convex, allowing faster convergence.

\begin{figure*}[!htbp]
        \subfloat[Average Loss]{%
            \includegraphics[width=.45\linewidth]{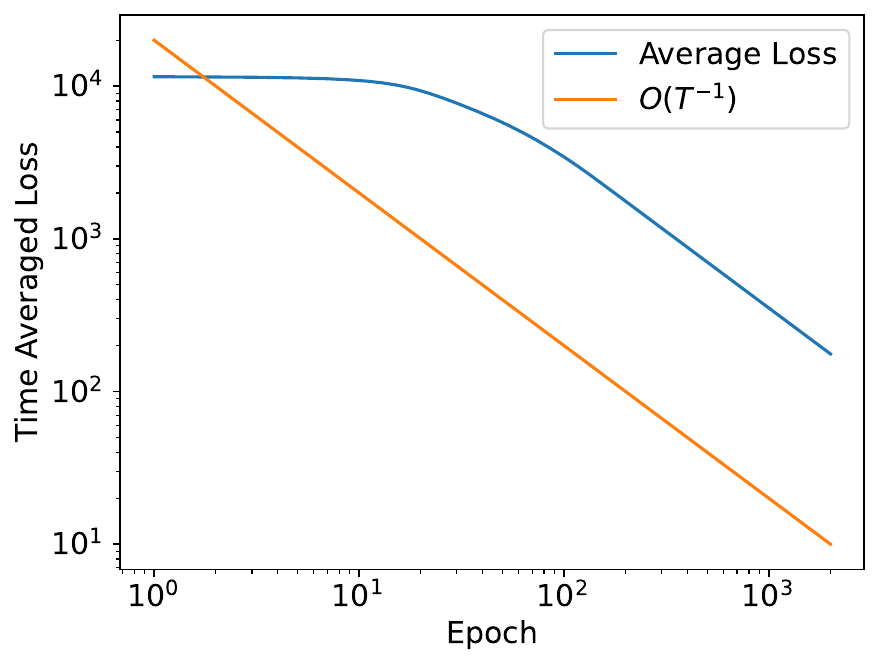}%
        }\hfill
        \subfloat[Average Gradient of Residuals]{%
            \includegraphics[width=.45\linewidth]{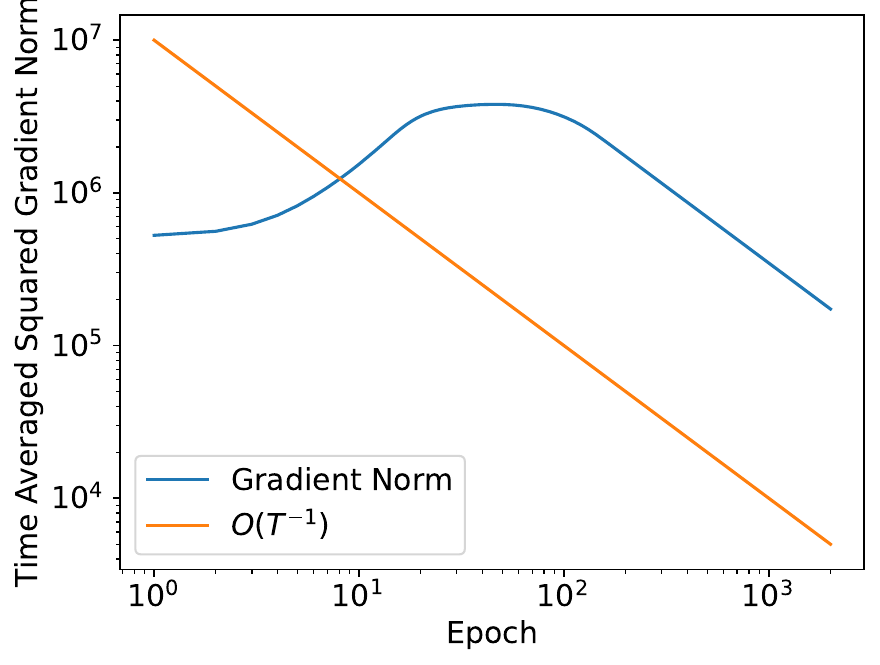}%
        }
        \caption{Convergence rates of the time averaged loss and the time averaged norm of gradient of the residuals.}
        \label{fig:convergence-experiment-time-averaged-quantities}
\end{figure*}

\subsection{Fast NTK Estimation}\label{ne:fast-ntk-estimation}

Here, we aim to demonstrate the behavior of the proposed algorithm for estimating the NTK and its trace, as well as our approximation of the NTK-based PINN weights.

\subsubsection{Quadratically parameterized predictor}\label{ne:quadratically-parameterized-predictor}
In this example, we verify the theoretical error rate that we obtained for the NTK approximation with respect to the number of samples taken. This example is a simple regression problem, i.e.\ not in the physics-informed setting, with a quadratically parameterized predictor. The setup is as follows: we take 50 equispaced values in $[-1, 1]$ denoted by $\{x_i\}_{i=1}^{50}$ and sample the data as
\[
    y_i = \pi x_i^2 + ex_i + \sqrt{2} + \xi_i,\quad \text{where}\quad \xi_i \overset{\text{iid}}{\sim} N(0, 1/\sqrt{2}),
\]
where the ``true" output is thus given by $$y = f(x) \doteq \pi x^2 + ex + \sqrt{2}.$$ We aim to learn an approximation $\widehat{f}$ of $f$ from the pairs $\{(x_i, y_i)\}_{i=1}^{50}$. By defining the feature map as $u(x) = (1, x, x^2)$, we can write $f(x) = (\theta^\ast\odot\theta^\ast)\cdot u(x)$, where $\theta^\ast = (\pi^{1/2}, e^{1/2}, 2^{1/4})$ and $\odot$ denotes element-wise multiplication. Our goal is to learn $\widehat{\theta} \in \mathbb{R}^3$ such that $\widehat{f}(x) = (\widehat{\theta}\odot\widehat{\theta})\cdot u(x) \approx f(x)$. 

Let us define the residual $\mathcal{R}(\theta)\in\mathbb{R}^{50}$ so that 
\[
    \mathcal{R}(\theta)_i = (\theta\odot\theta)\cdot u(x_i) - y_i, \quad \text{for}\quad i=1,2,\dots,50,
\]
so that can then compute $\nabla\mathcal{R}(\theta) \in \mathbb{R}^{3\times 50}$ exactly as 
\[
    \nabla\mathcal{R}(\theta)_{ji} = 2\theta_j u(x_i)_j.
\]
We note that we parameterized our predictor in this precise way so that $\nabla\calR(\theta)$ would depend on $\theta$. Indeed, an even simpler linearly parameterized predictor would result in a neural tangent kernel that is independent of $\theta$. The data, exact function $f$, and predictor $\widehat{f}$ found via Algorithm \ref{alg:moving-average-approx} can be seen in Figure \ref{fig:ntk-quadratic-experiment-data}. As expected, obtaining a reasonable estimator through gradient descent in this example is easy.

\begin{figure}[!htbp]
    \centering
    \includegraphics[width=0.45\linewidth]{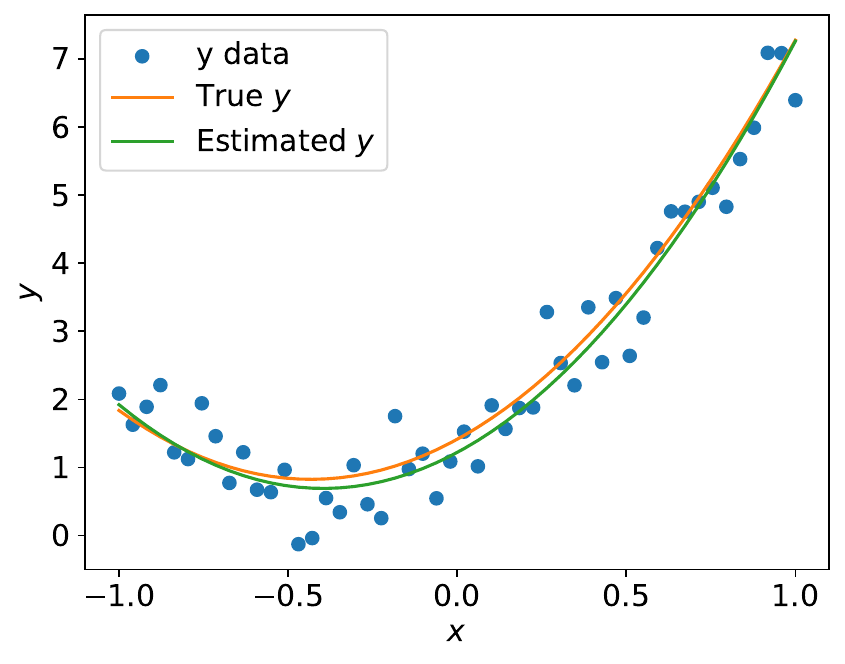}
    \caption{Quadratically parameterized predictor, sampled data, and true output.}
    \label{fig:ntk-quadratic-experiment-data}
\end{figure}

The main goal of this section is to verify the theoretical results obtained in Theorems \ref{thm:unbiased-estimates-up-to-discretization} and \ref{thm:Monte-Carlo-error}. To verify Theorem \ref{thm:unbiased-estimates-up-to-discretization}, we note that, by the law of large numbers, a Monte Carlo approximation of the expected value of our estimator at time 0 will be approximately the expected value of the estimator at that time:
\[
    \E\left[\widehat{K}(\theta_0)\right] \approx \frac{1}{N}\sum_{j=1}^N \widehat{K}^{(j)}(\theta_0).
\]
In Figure \ref{fig:ntk-quadratic-time-0-estimator}, we show a sequence of such estimators, where the final estimate with $N = 2\times 10^4$ samples exhibits an approximate NTK that pretty well matches the exact one at time 0. We note, however, that to obtain an accurate estimate of the NTK, we require a large number of samples. Despite the fact that, exploiting only one sample, the NTK estimate does not look like the exact one, we observe that the estimate with 2000 samples already provides a good qualitative approximation, even though it still contains some artifacts. To quantitatively explore this phenomenon, we now consider our theoretical results for the sequence of estimates.

\begin{figure*}[!htbp]
        \subfloat[Exact NTK at time 0]{%
            \includegraphics[width=.4\linewidth]{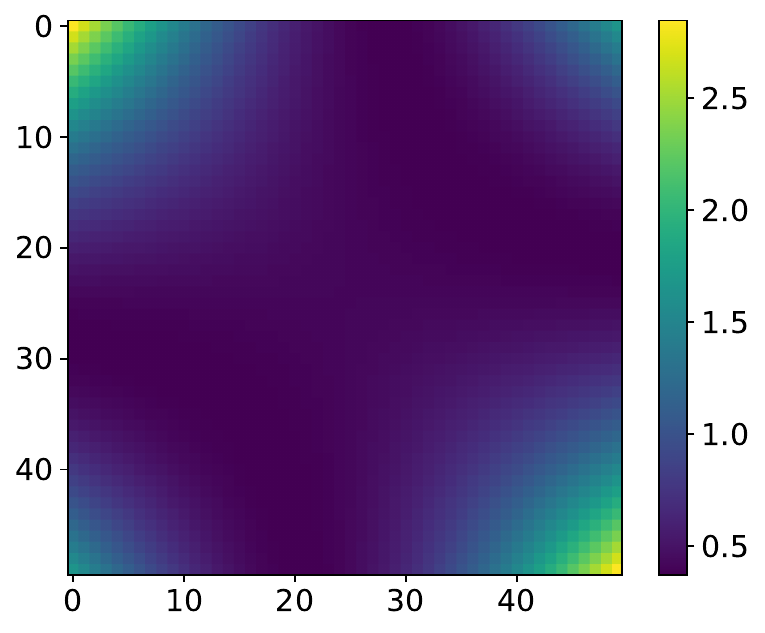}%
        }\hfill
        \subfloat[Estimate of NTK at time 0 (1 Sample)]{%
            \includegraphics[width=.4\linewidth]{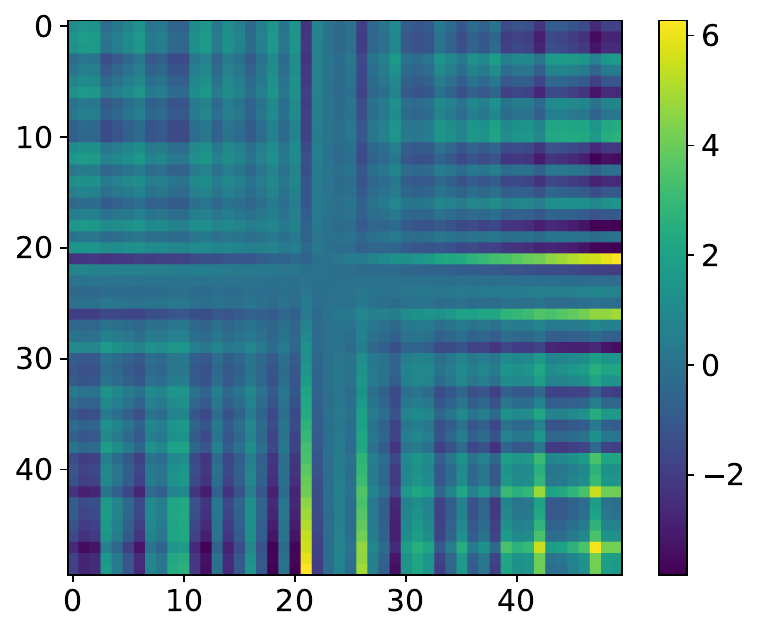}%
        }\\
        \subfloat[Estimate of NTK at time 0 (2000 Samples)]{%
            \includegraphics[width=.4\linewidth]{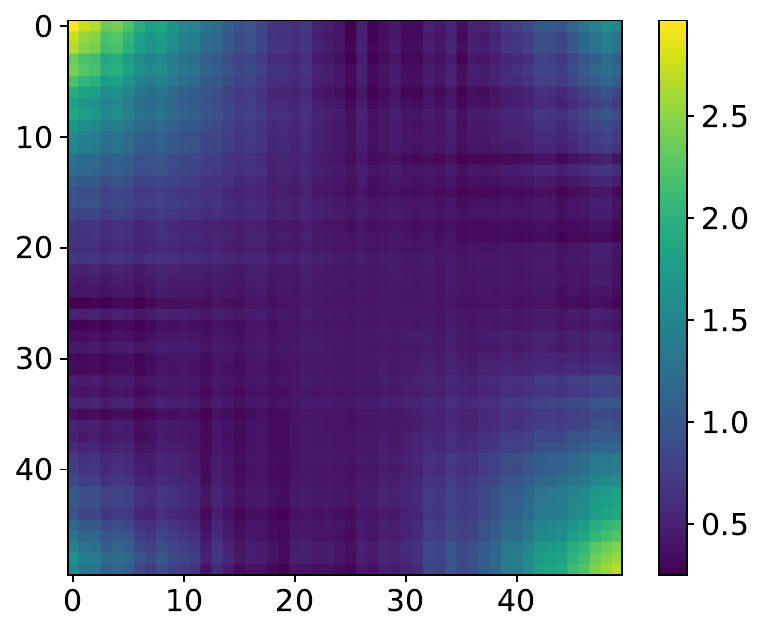}%
        }\hfill
        \subfloat[Estimate of NTK at time 0 (20000 Samples)]{%
            \includegraphics[width=.4\linewidth]{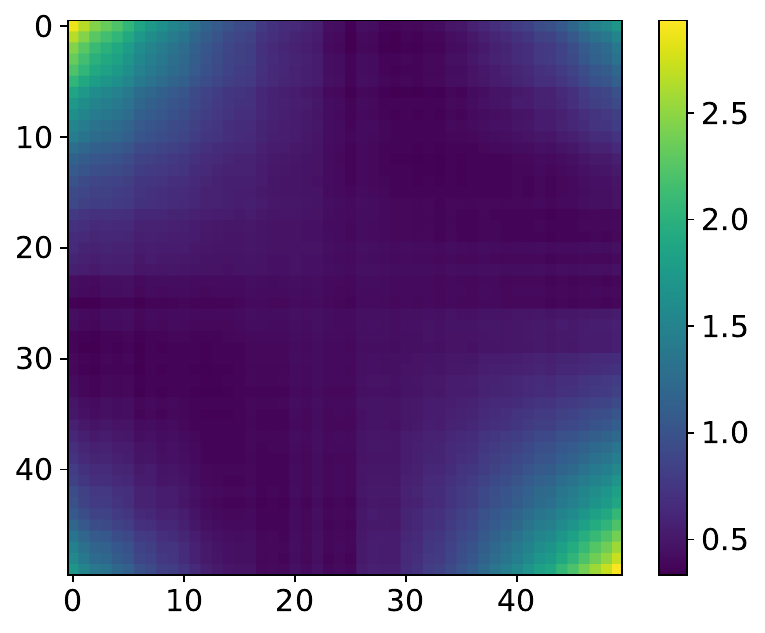}%
        }
        \caption{Approximation of NTK at a fixed time for differing numbers of samples.}
        \label{fig:ntk-quadratic-time-0-estimator}
\end{figure*}

To verify the Monte Carlo error rate in Theorem \ref{thm:Monte-Carlo-error}, we compute empirical averages of errors for $M=100$ independent Monte Carlo estimates of the NTK with varying $N$:
\[
    \E\left[\left\|\frac{1}{N}\sum_{j=1}^N\widehat{K}^{(j,N)}(\theta_0) - K(\theta_0)\right\|_F^2\right] \approx \frac{1}{M}\sum_{i=1}^M \left\|\frac{1}{N}\sum_{j=1}^N\widehat{K}^{(j,N)}(\theta_0) - K(\theta_0)\right\|_F^2,
\]
where the superscript $(j,N)$ denotes the fact that each $\widehat{K}^{(j,N)}$ is an i.i.d.\ sample from Algorithm \ref{alg:single-sample-approx} for each $j$ and $N$. An analogous approximation is done for the Monte Carlo trace estimator of the NTK. The results are shown in Figure \ref{fig:ntk-quadratic-monte-carlo-rate}, and, as expected, we observe a rate of the order $O(N^{-1})$ for the error, according to Theorem \ref{thm:Monte-Carlo-error}. We note that while the error may seem large, this is because the error is computed for the entire matrix with $2500$ entries, resulting in a small average error per entry.

\begin{figure*}[!htbp]
        \subfloat[NTK Monte Carlo error]{%
            \includegraphics[width=.48\linewidth]{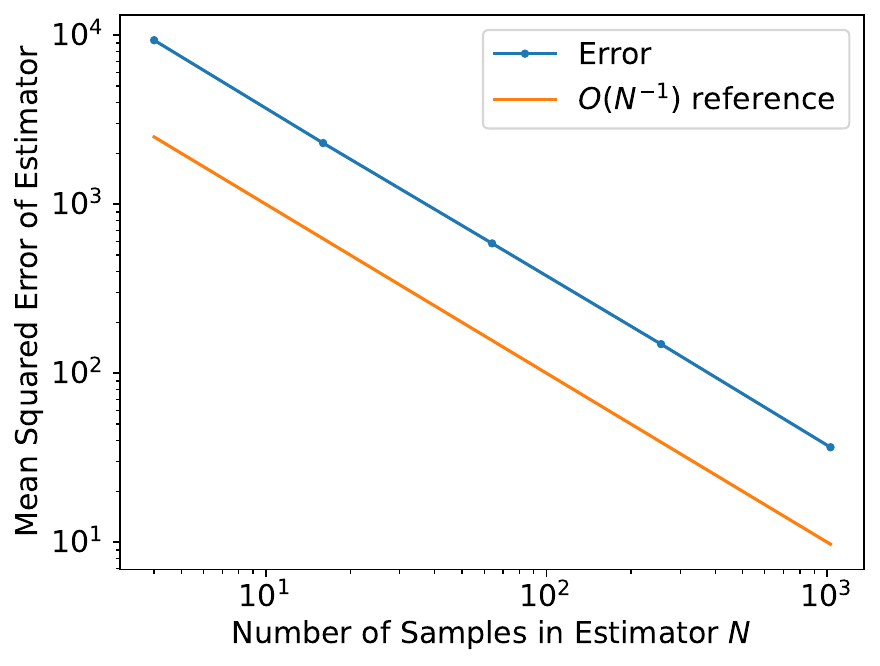}%
        }\hfill
        \subfloat[Trace NTK Monte Carlo error]{%
            \includegraphics[width=.48\linewidth]{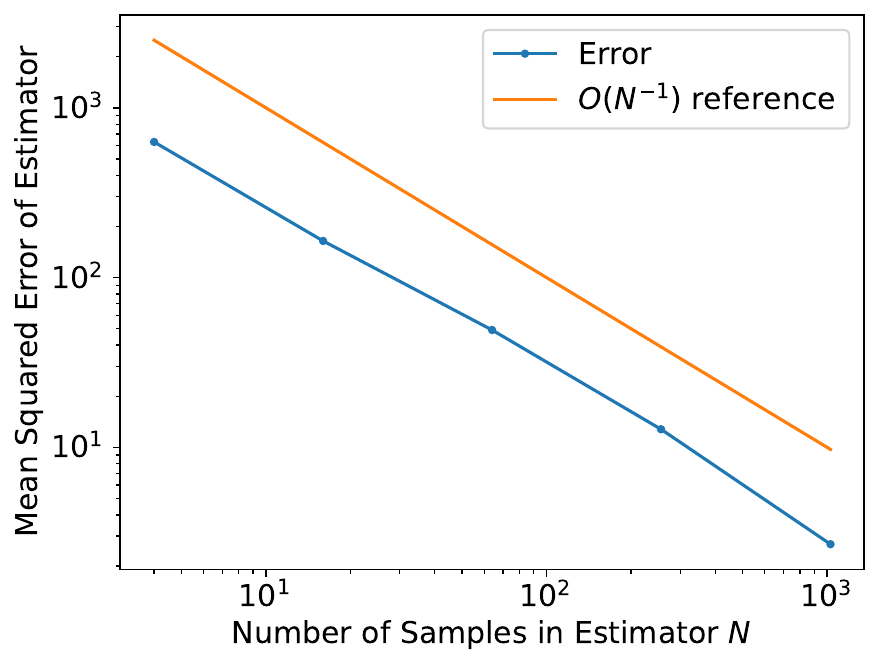}%
        }
        \caption{Monte Carlo error rates of efficient NTK and NTK trace estimators.}
        \label{fig:ntk-quadratic-monte-carlo-rate}
\end{figure*}

Finally, in Figure \ref{fig:ntk-quadratic-moving-average-comparison} we provide a comparison between the exact and estimated NTKs at the initial and final time $T = 10^5$ after applying Algorithm \ref{alg:moving-average-approx} with $\alpha = 10^{-4}$. We see that at the initial time, the upper right and lower left corners of the exact NTK have larger values than at the final time. The NTK approximation reflects this, showing that the moving average approximation of the NTK is able to adapt to a changing NTK over time. We will further demonstrate this adaptivity via the loss weights in the next experiments.

\begin{figure*}[!htbp]
        \subfloat[Exact NTK at time 0]{%
            \includegraphics[width=.4\linewidth]{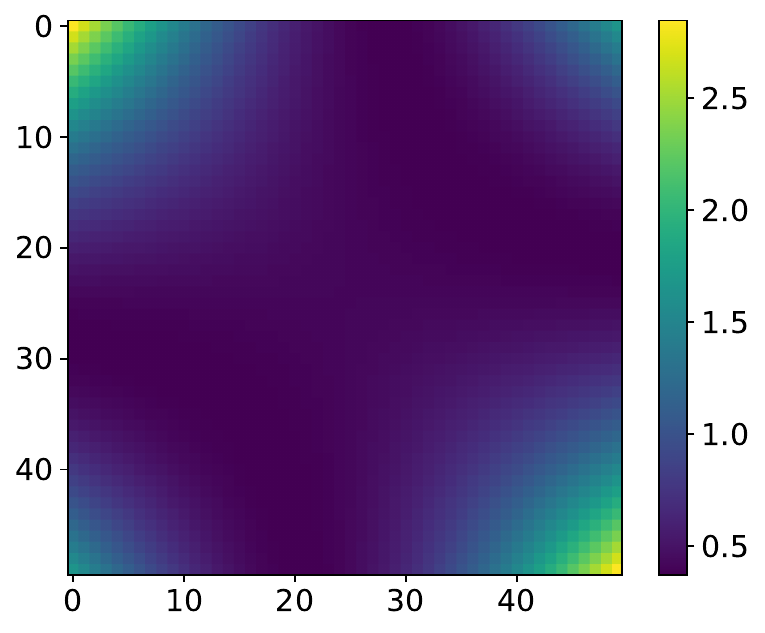}%
        }\hfill
        \subfloat[Estimate of NTK at time 0]{%
            \includegraphics[width=.4\linewidth]{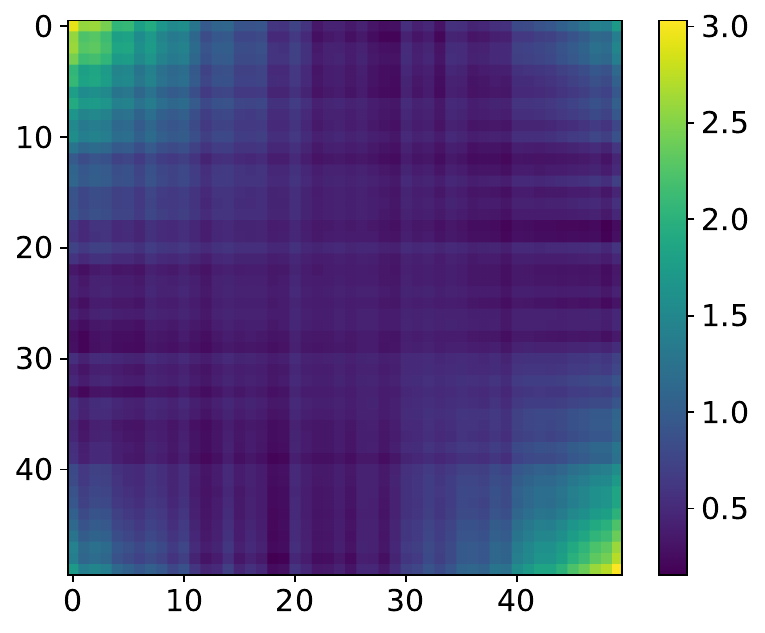}%
        }\\
        \subfloat[Exact NTK at final time]{%
            \includegraphics[width=.4\linewidth]{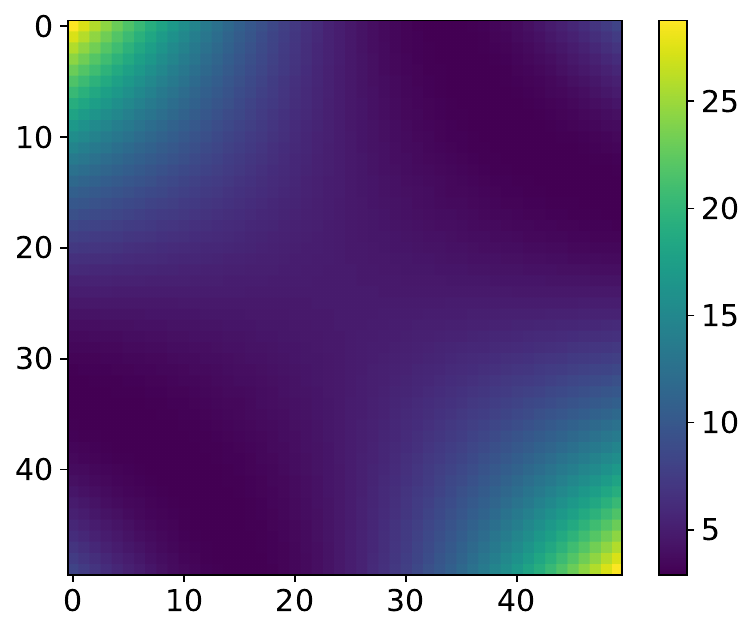}%
        }\hfill
        \subfloat[Estimate of NTK at final time]{%
            \includegraphics[width=.4\linewidth]{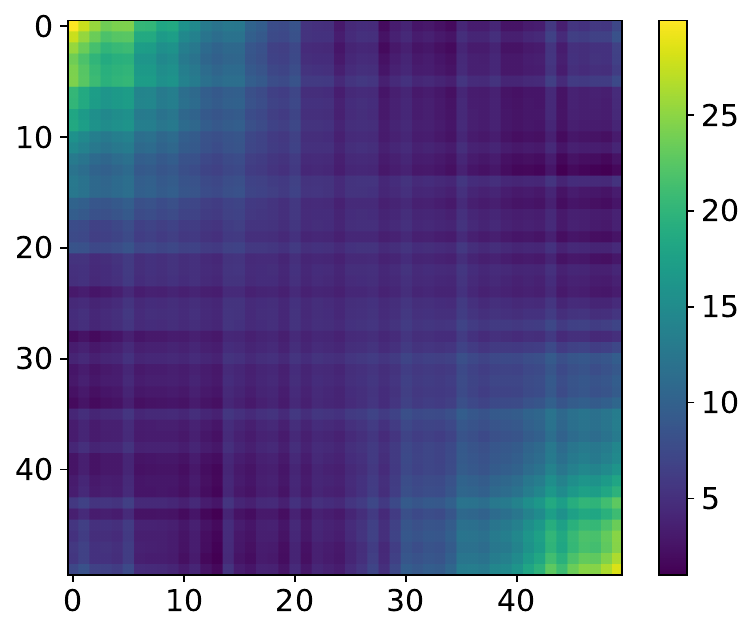}%
        }
        \caption{Approximation of NTK with evolution in time.}
        \label{fig:ntk-quadratic-moving-average-comparison}
\end{figure*}

\subsubsection{Wave Equation with PINNs}\label{ne:wave-equation-pinn}
Here, we aim to show the comparison between the efficient computation of the weights with respect to the exact NTK ones \eqref{eq:wang-inner-product}. We consider the following problem: find a neural network $u_\theta(t, x)$ with parameters $\theta$ that approximately solves the wave equation
\begin{equation}
    \begin{cases}
        u_{tt}(x, t) - 4u_{xx}(x,t) = 0, &(x, t) \in (0,1)\times(0,1), \\
        u(x, 0) = \sin(\pi x) + \frac12\sin(4\pi x), &x\in [0,1],\\
        u(0, t) = u(1, t) = 0, &t\in[0, 1],\\
        u_t(x, 0) = 0, &x\in[0,1].
    \end{cases}
\end{equation}
To obtain an accurate approximation of the physical phenomenon, within this benchmark, we further resample training points at each training iteration, using $n_D=300$ collocation points for the PDE residual, $n_{D_i}=300$ points in space for the initial time derivative, $n_{B_i}=100$ points in space for the initial condition, and $n_{B_1}=n_{B_2}=100$ points in space for each of the boundaries. Resampling is a common technique that was also used for the original presentation of the NTK-PINN loss weighting algorithm \cite{WangNTK2022}, considering a fixed training batch size sampled from the infinitely many points in the space-time domain. This leads to a smaller generalization error compared to using a fixed training set of the same size. So instead of fixing the points at which we compute the residuals to form the vector $\calR$, we now sample new points at each training step $t$, resulting in a newly defined residual denoted as $\calR_t(\theta)$. Thus, we now effectively approximate $\E[\nabla\calR_t(\theta_t)^\top\nabla\calR_t(\theta_t)]$, instead of $\nabla\calR_t(\theta_t)^\top\nabla\calR_t(\theta_t)$, and we use it in place of $K(\theta_t)$ in \eqref{eq:wang-inner-product}. By applying Algorithm \ref{alg:moving-average-approx} with $\alpha=10^{-3}$, we obtain the resulting loss weights depicted in Figure \ref{fig:wave-experiment-ntk-weights}. In general, we observed that the performance of the estimator is fairly robust to the choice of $\alpha$, and smaller $\alpha$ lead to smoother loss weight trajectories. In particular, $\alpha=10^{-4}$ is small enough to see the loss weight trajectories without being obscured by noise in the random estimates of the NTK.

\begin{figure}[!htbp]
    \centering
    \includegraphics[width=0.65\linewidth]{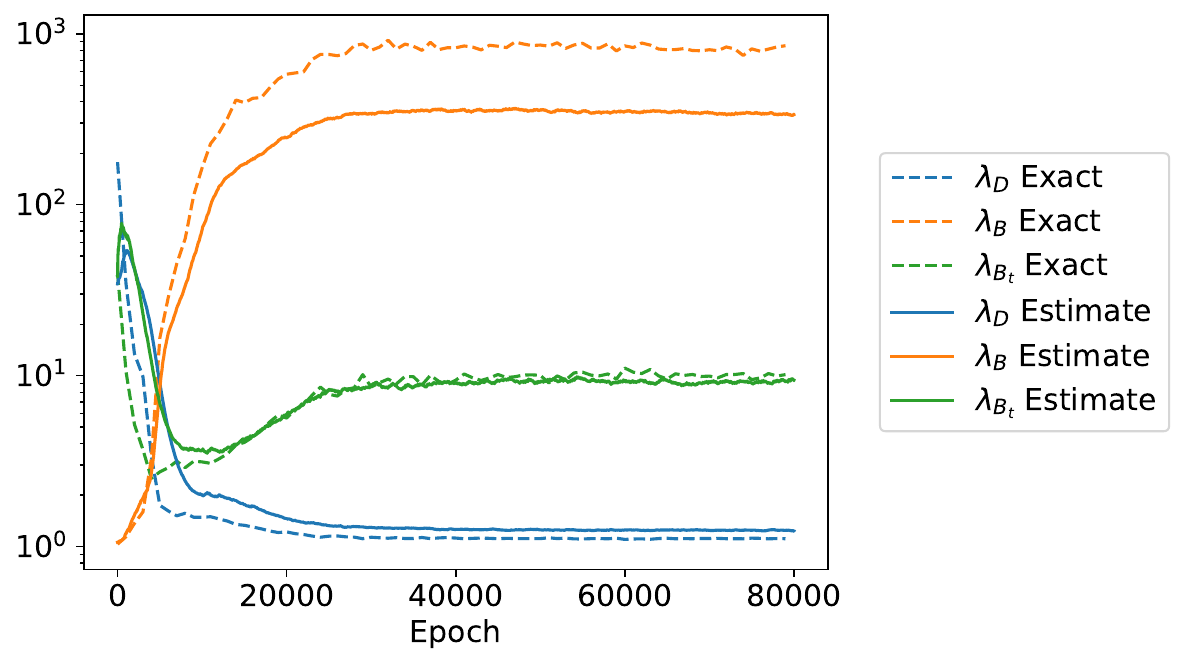}
    \caption{Exact NTK-based loss weights (shown with dashed lines) given by \eqref{eq:wang-inner-product}, and the loss weights based on the approximated NTK (shown with solid lines) given by Algorithm \ref{alg:moving-average-approx} combined with resampled data.}
    \label{fig:wave-experiment-ntk-weights}
\end{figure}

We see that the exact and approximated loss weights exhibit a very similar behavior, though the largest loss weight is under-approximated, while the smallest is over-approximated. Note that because the loss weights involve the quotient of random variables, even though our trace estimator is unbiased up to a negligible discretization error, the results obtained by replacing $K$ in \eqref{eq:wang-inner-product} with its randomized approximations are not unbiased. Despite this discrepancy, the resulting PINN approximation to the wave equation, as shown in Figure \ref{fig:wave-equation-solution}, has a low error compared to the exact solution given by
\[
    u(x,t) = \sin(\pi x)\cos(2\pi t) + \frac12 \sin(4\pi x)\cos(8\pi t).
\]
To obtain these results, we used a feedforward neural network with three hidden layers and 500 neurons each, Xavier initialization, and the hyperbolic tangent as the activation function.

\begin{figure}[!htbp]
    \centering
    \includegraphics[width=0.95\linewidth]{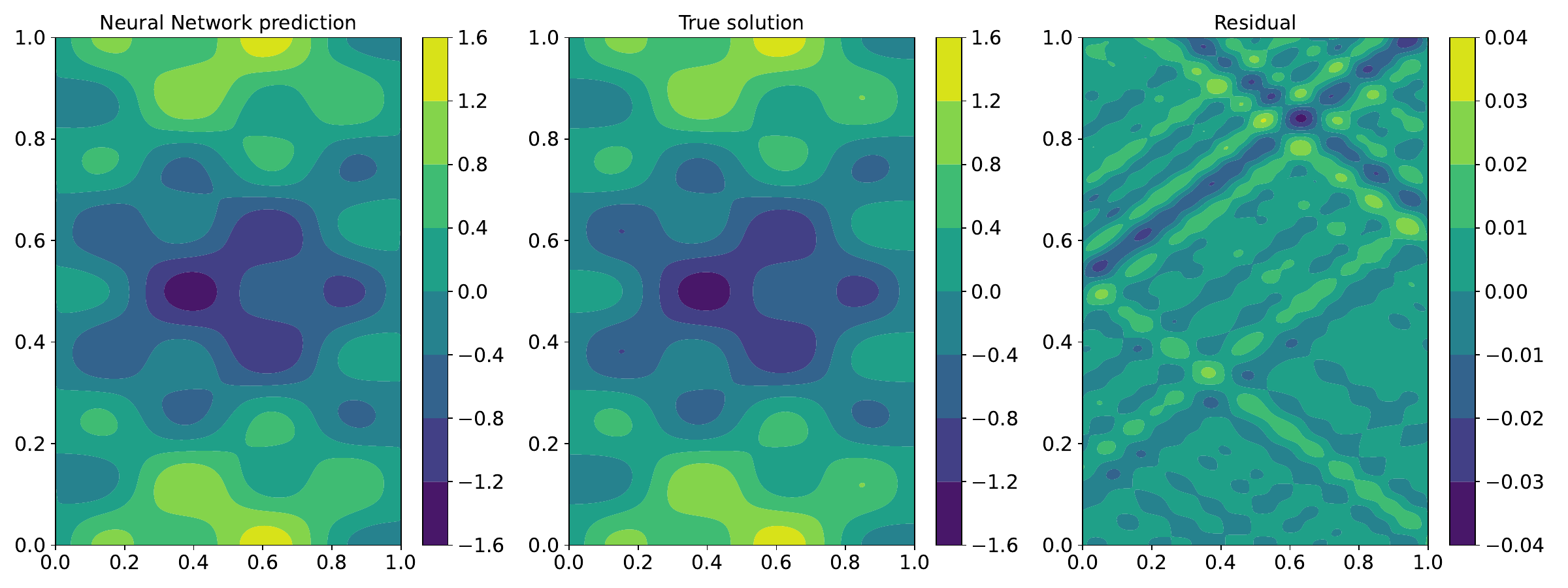}
    \caption{Wave equation solution obtained using the loss weighting Algorithm \ref{alg:moving-average-approx} combined with resampled data.}
    \label{fig:wave-equation-solution}
\end{figure}

\subsubsection{Liquid Crystal Q-Tensor PINN}\label{ne:q-tensor}
In this final example, we show the efficacy of our weighting algorithm for a nonlinear PDE which describes the evolution of liquid crystal molecule orientations. The model in two spatial dimensions is given by
\[
    \begin{cases}
        Q_t = L\Delta Q - aQ - c\Tr(Q^2)Q, &(\x, t) \in \Omega\times[0,T],\\
        Q(\x, 0) = Q_0(\x), & \x\in\Omega,\\
        Q(\x, t) = 0, &(t, \x) \in \partial\Omega\times[0,T],
    \end{cases}
\]
where $\Omega\subseteq\R^2$ and $Q:\overline{\Omega}\times[0,T]\to \R^{2\times 2}$ is a trace-free and symmetric matrix-valued function referred to as the Q-tensor \cite{gudibanda-weber-yue,hirsch-weber-yue}. Additionally, $a\in\R$ and $L,c>0$ are parameters that characterize the behavior of the solution, and because $Q$ is trace-free and symmetric, there are only two independent variables in this equation. The initial condition is given by
\[
    Q_0(x, y) = dd^\top - \frac{\|d\|^2}{2}I,\quad d = \frac{\begin{bmatrix} (2-x)(2+x)(2-y)(2+y) & \sin(\pi x)\sin(\pi y)\end{bmatrix}^\top}{\sqrt{1 + [(2-x)(2+x)(2-y)(2+y)]^2 + \sin(\pi x)^2\sin(\pi y)^2}}.
\]

The eigenvectors of the function $Q$ encode the average direction of the liquid crystal molecules at a point in space, and depending on the direction of the molecules, light may or may not pass through the liquid crystal. The light passing through the liquid crystal can be seen by a contour plot of $|\lambda_1(\x,t)(v_1(\x,t)\cdot(1,0))|$, where $v_1(\x,t)$ is the eigenvector corresponding to the positive eigenvalue of $Q(\x,t)$. 

As in the previous example, we train a PINN model $Q_\theta$ with our adaptive weighting algorithm to approximate the solution to the Q-tensor PDE, using the same feedforward neural network with three hidden layers and 500 neurons each. We sample $n_D=n_{B_i}=300$ data points each for the interior residual and for the initial condition, while for each portion of the boundaries we use $n_{B_k}=10$ data samples for $k=1,2,3,4$, still resampling the points at each iteration of training. 

In Figure \ref{fig:q-tensor-solution}, we plot the resulting PINN (denoted NTK-PINN) compared to a reference Finite Element solution computed with the invariant quadratization method \cite{hirsch-weber-yue}, as well as compared to a PINN trained without the adaptive weighting algorithm (denoted PINN). The contours show the aforementioned visualization of light passing through the liquid crystal, while the blue vectors represent the eigenvectors $v_1(\x,t)$. We see that all three solutions are qualitatively very similar, providing an accurate reconstruction of the contour plots and of the eigenvectors reaching the horizontal equilibrium state at time $t=2$. We note that the artifacts around the boundary for the NTK-PINN and PINN solutions are due to the fact that $Q=0$ on the boundary, so computing the eigenvectors will be sensitive to any error, while for the Finite Element solution, $Q$ is constrained to equal 0 exactly on the boundary. Thus, even though the use of the adaptive loss weights makes little difference in the quality of solution for this example, we emphasize that the Algorithm \ref{alg:moving-average-approx} can offer enough improvement in computational efficiency to make such an adaptive NTK-based loss weighting feasible.

\begin{figure}[!htbp]
\begin{center}
\begin{tikzpicture}
    \draw (5.1, 3) node {\textbf{FEM}};
    \draw (0, 3) node {\textbf{NTK-PINN}};
    \draw (-5, 3) node {\textbf{PINN}};
    \draw (5.6, 0) node[inner sep=0] {\includegraphics[scale=0.416]{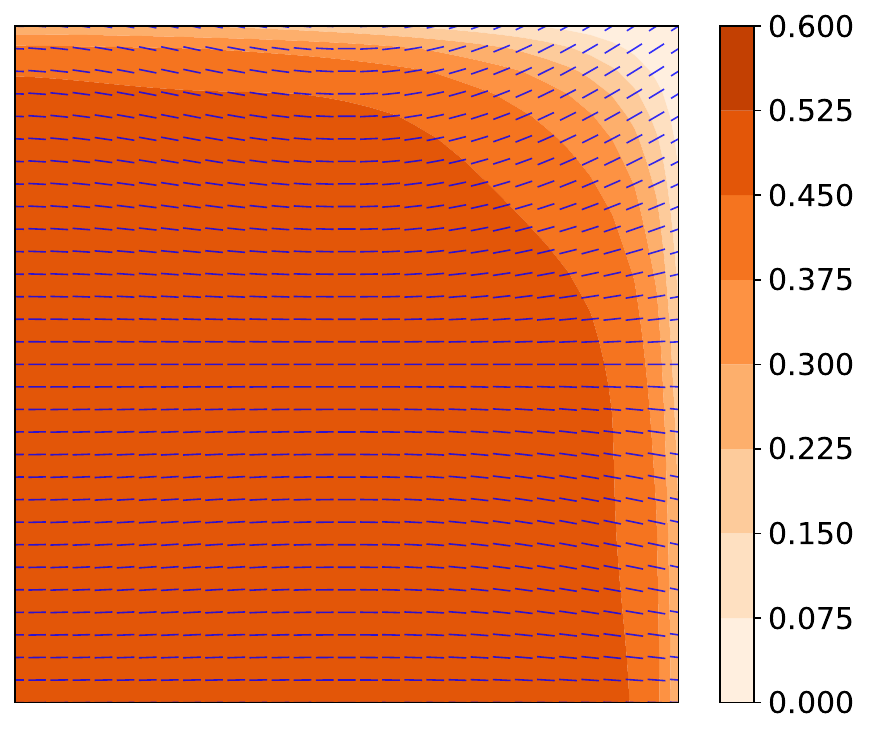}};
    \draw (0, 0) node[inner sep=0] {\includegraphics[scale=0.4]{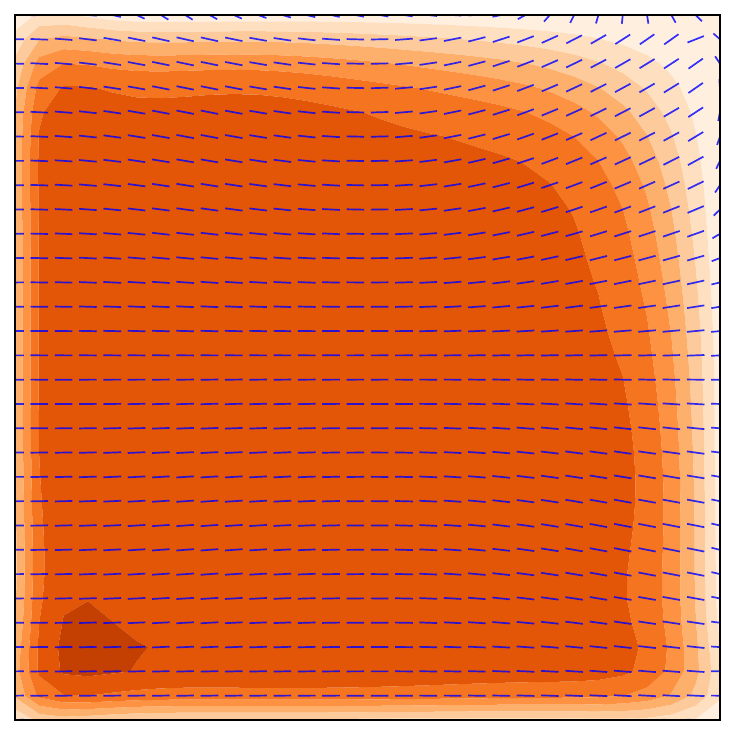}};
    \draw (-5, 0) node[inner sep=0] {\includegraphics[scale=0.4]{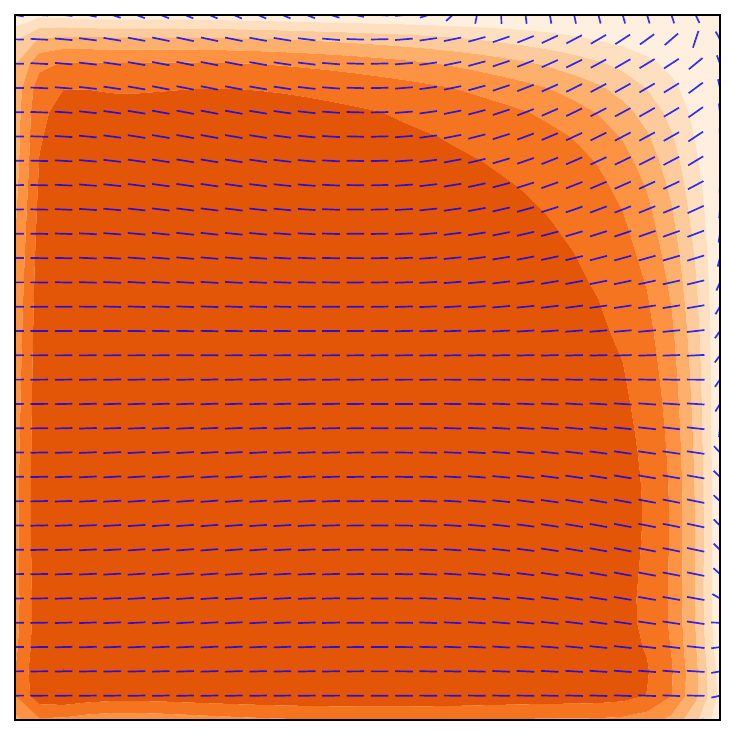}};
    \draw (0, -2.75) node {Solutions at $t=0$};
    \draw (5.6, -6) node[inner sep=0] {\includegraphics[scale=0.416]{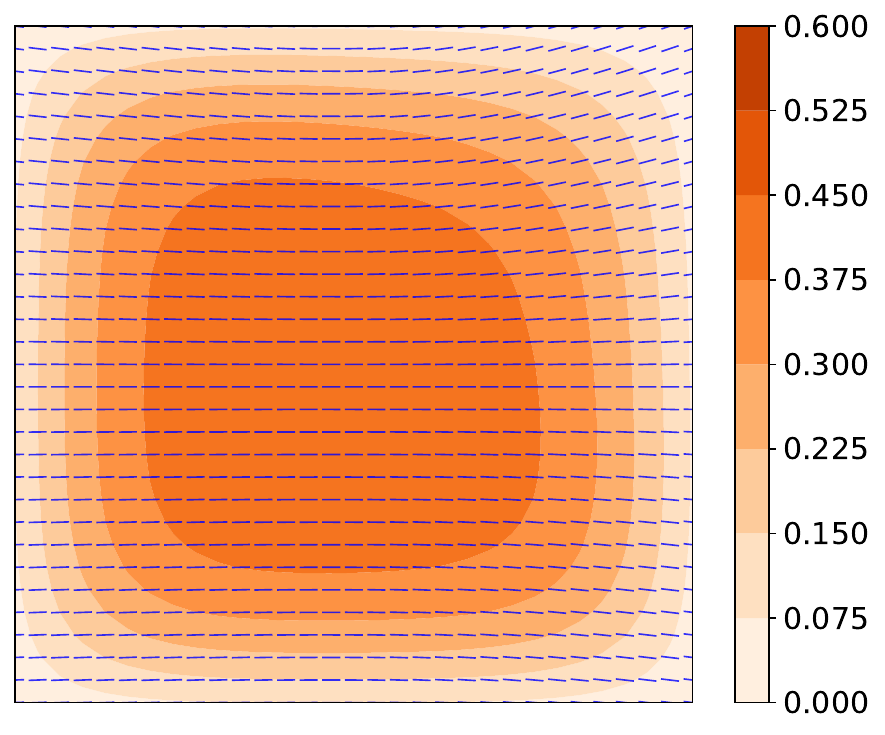}};
    \draw (0, -6) node[inner sep=0] {\includegraphics[scale=0.4]{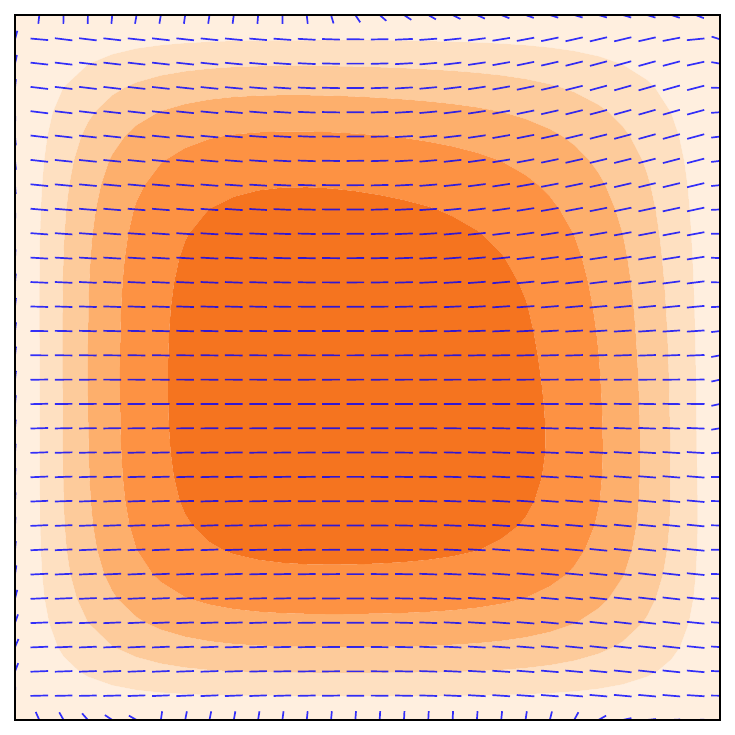}};
    \draw (-5, -6) node[inner sep=0] {\includegraphics[scale=0.4]{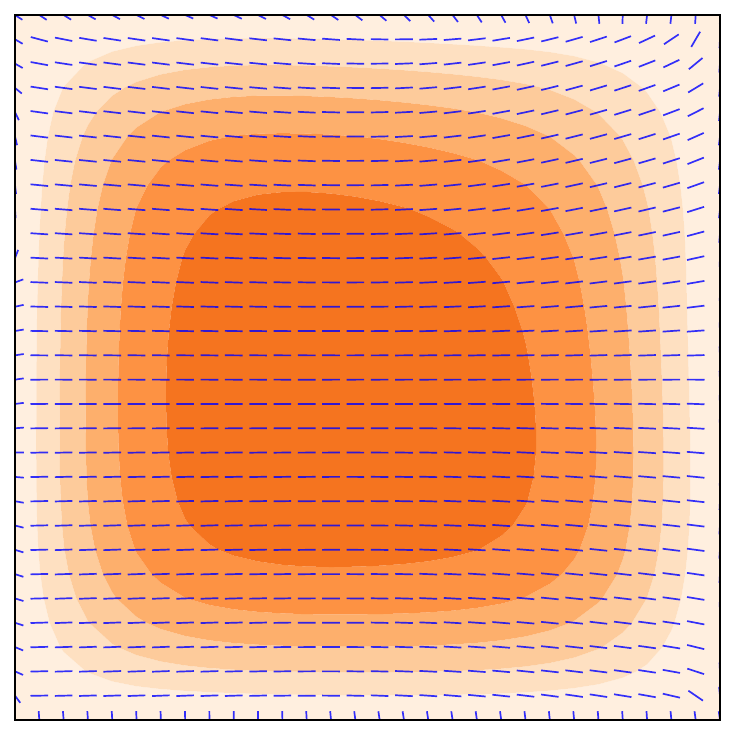}};
    \draw (0, -8.75) node {Solutions at $t=0.5$};
    \draw (5.6, -12) node[inner sep=0] {\includegraphics[scale=0.416]{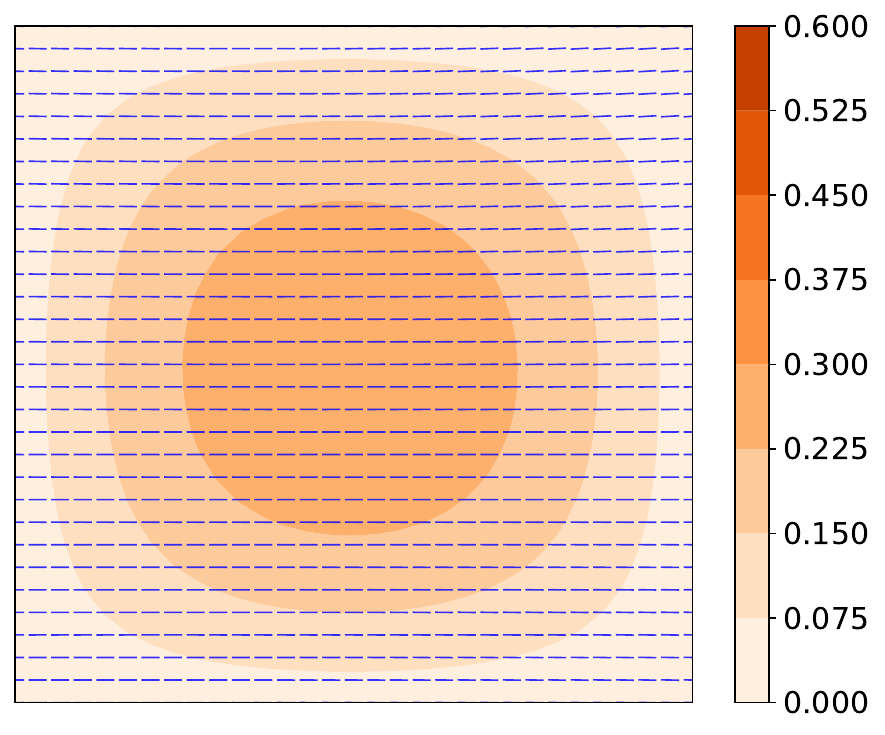}};
    \draw (0, -12) node[inner sep=0] {\includegraphics[scale=0.4]{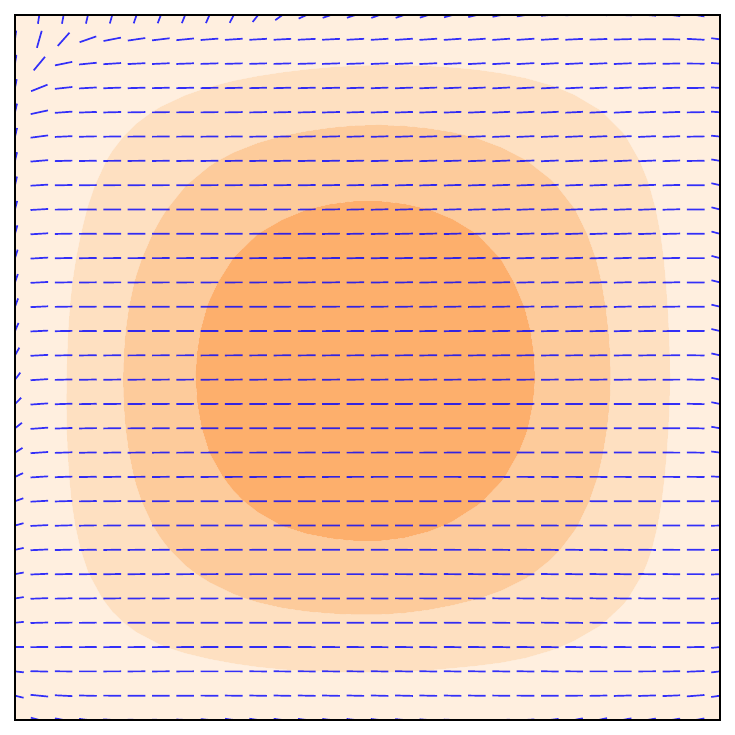}};
    \draw (-5, -12) node[inner sep=0] {\includegraphics[scale=0.4]{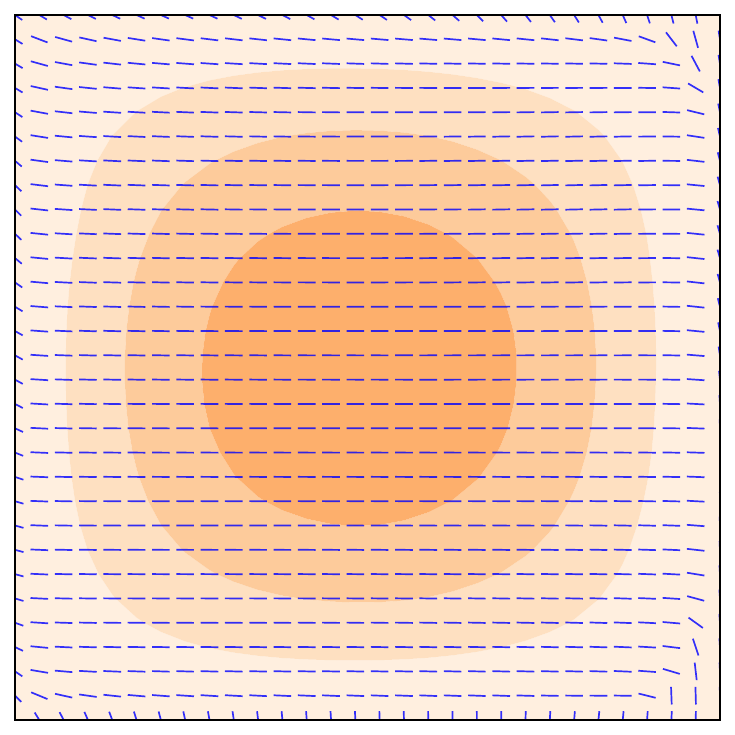}};
    \draw (0, -14.75) node {Solutions at $t=2$};
    \end{tikzpicture}
\end{center}
\caption{Finite Element solution (FEM) compared to a PINN solution trained with the approximate NTK weight algorithm (NTK-PINN) and a PINN solution trained without adaptive weighting (PINN).}
\label{fig:q-tensor-solution}
\end{figure}
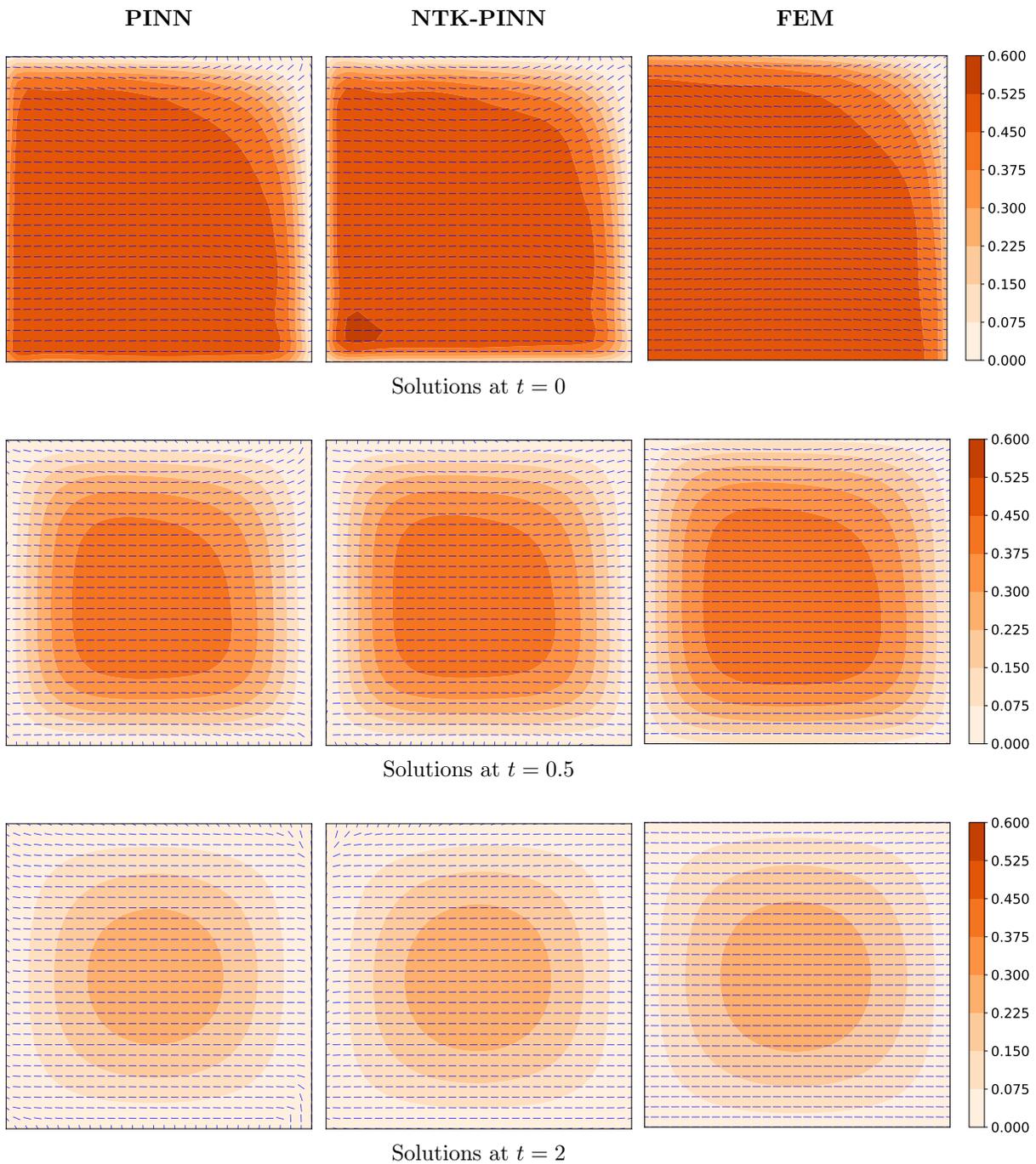

\section{Conclusion}
In this paper, we presented a convergence analysis for a neural tangent kernel-based loss weighting scheme in multi-objective optimization with a physics-informed loss. We showed that, under certain assumptions, the average squared norms of the residuals and of the loss gradients converge to zero. To address concerns with computational complexity, we further proposed a randomized algorithm which allowed us to obtain sketches of the NTK. This way, we demonstrated that, using a moving average, computing adaptive NTK approximations only costs one additional evaluation and backpropagation of our physics-informed neural network per iteration. We provided error estimates to justify the proposed randomized algorithm, and we verified our theoretical results numerically. We performed further numerical tests demonstrating the efficacy of our methodology for the wave equation and a nonlinear model of liquid crystal dynamics. In the future, we would also like to study the convergence of PINN training in the infinite width limit, particularly for nonlinear PDEs. Another interesting open question is related to the discovery of the optimal weights used to balance the loss functions. Lastly, an important application of our work will be to physics-informed machine learning for reduced order modeling. In this context, the neural networks used are vector-valued, and computing the loss weights based on the NTK is particularly expensive, so we envision our NTK approximation algorithm having great benefits in that context \cite{ChenHesthaven2021,HirschNeuralEmpiricalInterpolation2024}.

\section*{Acknowledgments}

This material is based upon work supported by the National Science Foundation Graduate Research Fellowship Program under Grant No.\ DGE 2146752. Any opinions, findings, and conclusions or recommendations expressed in this material are those of the authors and do not necessarily reflect the views of the National Science Foundation. \textbf{FP} acknowledges the support provided by the European Union - NextGenerationEU, in the framework of the iNEST - Interconnected Nord-Est Innovation Ecosystem (iNEST ECS00000043 - CUP G93C22000610007) consortium and its CC5 Young Researchers initiative. \textbf{FP} also acknowledges INdAM-GNSC and the project ``Sviluppo e analisi di modelli di ordine ridotto basati su tecniche di deep learning" (CUP E53C24001950001) for its support.

\appendix

\section{Proof of Descent Lemma \ref{lem:descent-lemma}}\label{sec:proof-descent-lemma}
\begin{proof}
    Let $g(t) = F(y + t(x-y))$, which is well-defined since $\calP$ is convex. Then $g'(t) = \langle \nabla F(y+t(x-y)), x-y\rangle$. By the fundamental theorem of calculus, we have
    \[
        F(x) - F(y) = g(1) - g(0) = \int_0^1 g'(t)\, dt = \int_0^1 \langle\nabla F(y+t(x-y)), x-y\rangle\, dt.
    \]
    Subtract $\langle\nabla F(y), x-y\rangle$ from both sides of this equation and take absolute values to obtain
    \begin{align*}
        \lvert F(x) - F(y) - \langle \nabla F(y), x-y\rangle \rvert &\le \int_0^1 \lvert \langle \nabla F(y+t(x-y)) - \nabla F(y), x-y\rangle\rvert\, dt\\
        &\le \int_0^1 \|\nabla F(y+t(x-y)) - \nabla F(y)\|\cdot\|x-y\|\, dt\\
        &\le \int_0^1 Lt\|x-y\|^2\, dt\\
        &= \frac{L}{2}\|x-y\|^2,
    \end{align*}
    which completes the proof.
\end{proof}

\section{Proof of Theorem \ref{thm:alternative-ntk-approximation}}\label{sec:proof-alternative-ntk-approximation}

\begin{proof}
    Performing a Taylor expansion, we have
    \begin{align*}
        \E\left[\|\calR(\theta_t + g_t) - \calR(\theta_t)\|_F^2\right]
        &= \E\left[\|\calR(\theta_t) + \nabla \calR(\theta_t)^\top g_t + \mathcal{E}(\theta_t,g_t) - \calR(\theta_t)\|_F^2\right]\\ 
        &= \E\left[\|\nabla \calR(\theta_t)^\top g_t + \mathcal{E}(\theta_t,g_t)\|_F^2\right],
    \end{align*}
    where we define $\mathcal{E}(\theta_t, g_t)$ to be the Taylor remainder, which by Assumption \ref{assumption:bounded-residual-derivatives} satisfies
    \[
        \|\mathcal{E}(\theta_t, g_t)\|_F \le C\varepsilon^2\|g_t/\varepsilon\|^2.
    \]
    Combining this with the fact that $\E[\|g_t/\varepsilon\|^3] \le \E[\|g_t/\varepsilon\|^4] = n(n+2)$, since $g_t/\varepsilon$ is a standard Gaussian vector, gives
    \[
        \E\left[\|g_t\|\|\mathcal{E}(\theta_t, g_t)\|_F\right] = \varepsilon\E\left[\|g_t/\varepsilon\|\|\mathcal{E}(\theta_t, g_t)\|_F\right] \le C\varepsilon^3 n(n+2),\quad \text{and} \quad \E\left[\|\mathcal{E}(\theta_t, g_t)\|_F^2\right] \le C^2\varepsilon^4 n(n+2).
    \]
    Again, by Assumption \ref{assumption:bounded-residual-derivatives}, it follows that 
    \begin{align*}
        \E\left[\|\calR(\theta_t + g_t) - \calR(\theta_t)\|_F^2\right] 
        &= \E\left[\|\nabla\calR(\theta_t)^\top g_t\|_F^2 + 2\langle \nabla\calR(\theta_t)^\top g_t, \mathcal{E}(\theta_t, g_t)\rangle + \|\mathcal{E}(\theta_t, g_t)\|_F^2\right]\\
        &\le \E\left[\|\nabla\calR(\theta_t)^\top g_t\|_F^2 + 2C\|g_t\|\|\mathcal{E}(\theta_t, g_t)\|_F + \|\mathcal{E}(\theta_t, g_t)\|_F^2\right]\\
        &\le \E\left[\|\nabla\calR(\theta_t)^\top g_t\|_F^2\right] + 2C\varepsilon^3 n(n+2) + C^2\varepsilon^4 n(n+2),
    \end{align*}
    and similarly,
    \[
        \E\left[\|\calR(\theta_t + g_t) - \calR(\theta_t)\|_F^2\right]  \ge \E\left[\|\nabla\calR(\theta_t)^\top g_t\|_F^2\right] - 2C\varepsilon^3 n(n+2).
    \]
    Now, using the cyclic property of the trace, we have that
    \[
        \E\left[\|\nabla\calR(\theta_t)^\top g_t\|_F^2\right] = \E\left[\Tr\left(g_t^\top\nabla\calR(\theta_t)\nabla\calR(\theta_t)^\top g_t\right)\right] = \Tr(\nabla\calR(\theta_t)^\top \E[g_t g_t^\top]\nabla\calR(\theta_t)) = \varepsilon^2 \Tr(K(\theta_t)).
    \]
    Thus, we obtain
    \[
        -2C\varepsilon n(n+2) \le \E\left[\left\|\frac{\calR(\theta_t + g_t) - \calR(\theta_t)}{\varepsilon}\right\|_F^2\right] - \Tr(K(\theta_t)) \le 2C\varepsilon n(n+2) + C^2\varepsilon^2 n(n+2),
    \]
    and the result follows from $n(n+2) < (n+1)^2$ and by bounding from above with a larger suitable constant.
\end{proof}

\bibliographystyle{abbrv}
\bibliography{bibliography}
\end{document}